\documentclass[11pt]{article}

\usepackage[latin1]{inputenc}
\usepackage{epsfig}
\usepackage{graphicx,psfrag}
\usepackage[tight]{subfigure} 
\usepackage{amsmath}
\usepackage{amsfonts,amsthm,bbm,amssymb,epsf,amsmath,%
latexsym,amscd,amsfonts,enumerate,amsthm,supertabular,color,url,multirow}

\usepackage[colorlinks,bookmarksopen,bookmarksnumbered,citecolor=blue,urlcolor=blue]{hyperref}

\newtheorem{theorem}{Theorem}[section]
\newtheorem{lemma}[theorem]{Lemma}
\newtheorem{corollary}[theorem]{Corollary}
\newtheorem{proposition}[theorem]{Proposition}

\newtheorem*{keywords}{Keywords}

\newenvironment{romanlist}{

	\let\item\Item
	\begin{enumerate}
	}{
	\end{enumerate}}
	
	\let\Item\item

\def\address#1{\expandafter\def\expandafter\@aabuffer\expandafter
	{\@aabuffer{\affiliationfont{#1}}\relax\par
	\vspace{13pt}}}

\numberwithin{equation}{section}

\newcommand{\RR}{\mathbb{R}}
\newcommand{\ZZ}{\mathbb{Z}}

\newcommand{\CC}{\mathbb{C}}
\newcommand{\TT}{\mathbb{T}}

\def\cP{\mathcal{P}}
\def\cL{\mathcal{L}}

\def\cA{\mathcal{A}}

\def\cH{\mathcal{H}}

\DeclareMathOperator{\Ker}{Ker} 
\DeclareMathOperator{\supp}{supp} 
\DeclareMathOperator{\Ree}{Re} %
\DeclareMathOperator{\Imm}{Im} %
\newcommand{\noopsort}[1]{} %

\begin{document}

\title{
Construction of Bases in Modules over Laurent Polynomial Rings and Applications to Box Spline Prewavelets}%

\author{Oleg Davydov\thanks{Department of Mathematics,
Justus Liebig University  Giessen, Arndtstrasse 2, 35392 Giessen, Germany,
\url{oleg.davydov@math.uni-giessen.de}} \quad 
and\quad Anatolii Tushev\thanks{Department of Mathematics,
Justus Liebig University  Giessen, Arndtstrasse 2, 35392 Giessen, Germany,
\url{anavlatus@gmail.com}
}}

\date{August 2, 2025}

\maketitle

\begin{abstract}
We suggest a new method of basis construction for the kernel of a linear form on 
the Laurent polynomial module related to multivariate wavelets, and demonstrate its 
applications to box spline prewavelets, leading to small mask supports for $C^1$ cubic and $C^2$ quartic box splines in two
variables, outperforming previously known constructions,
and to  trivariate piecewise linear prewavelets with at most 23 nozero mask coefficients.
 \begin{keywords}
Wavelets, Laurent polynomials, basis of a module, box splines
\end{keywords}
\end{abstract}

\section{Introduction}\label{intro}

Recall that the scaling function $\varphi:\RR^d\to\CC$ 
of a multiresolution analysis in $d$ variables 
(see e.g.~\cite{Meyer93,KPS16}) belongs to
$L_2(\RR^d)$ and satisfies a refinement equation 
\begin{equation}\label{refeq}
\varphi(x)=\sum_{k\in\ZZ^d}h_k \varphi(2x-k),\qquad h_k\in\CC,\quad x\in\RR^d,
\end{equation}
for a certain coefficient sequence $h\in\ell_2(\ZZ^d)$,
which can be expressed in terms of the Fourier transform $\hat\varphi(u):=\int_{\RR^d}\phi(x)e^{-iu^T\!x}\,dx$ as
\begin{equation}\label{refeqH}
\hat \varphi(2u)=H(z)\hat \varphi(u),\qquad u\in\RR^d,\quad z=e^{-iu}\in\TT^d,
\end{equation}
$$
\TT^d:=\{z\in\CC^d:|z_j|=1,\;j=1,\ldots,d\},$$
where $H$ is a Laurent series 
\begin{equation}\label{Hser}
H(z)=2^{-d}\sum_{k\in\ZZ^d}h_kz^k,\qquad\text{where}\quad z^k:=z_1^{k_1}\cdots z_d^{k_d}\in\CC^d.
\end{equation}

The Laurent series in \eqref{Hser} converges on $\TT^d$ in $L_2(\TT^d)$-norm as a Fourier series with square summable
coefficients. Denote by $\cL(\ell_p)$ the vector space
$$
\cL(\ell_p)=\Big\{\sum_{k\in\ZZ^d}a_kz^k:a\in \ell_p(\ZZ^d)\Big\},\qquad 1\le p\le\infty.$$
If $H\in \cL(\ell_1)\subset\cL(\ell_2)$, then the Fouries series is absolutely convergent and hence $H$ 
is continuous on $\TT^d$. Note that $\cL(\ell_1)$ is a commutative ring since 
the product of two absolutely convergent Fourier series is again an absolutely convergent Fourier series.

Note that $\bar z=z^{-1}$ for $z\in \TT^d$. Even though we prefer to consider the elements $P\in \cL(\ell_p)$ as Laurent series for
the sake of convenience of algebraic computations, we only need them for $z\in \TT^d$, where they are Fourier  series with
respect to $u$. Therefore we define the \emph{conjugation} operator by
\begin{equation}\label{conj}
\overline{P(z)}:=\sum_{k\in\ZZ^d}\overline{a_k}z^{-k}\quad\text{for}\quad P(z)=\sum_{k\in\ZZ^d}a_kz^k,\quad a_k\in\CC.
\end{equation}

If $\varphi$ has a compact support, then only finite many of $h_k$ are nonzero, and hence $H$ 
is a Laurent polynomial of a complex variable $z\in\CC^d$ with real coefficients
in the commutative ring
$$
\cL=\Big\{\sum_{k\in M}a_kz^k:M\subset\ZZ^d \text{ finite, } a_k\in\CC\Big\}.$$

The sequence of multiresolution spaces $V_j\subset L_2(\RR^d)$, $j\in\ZZ$, is obtained by
\begin{align*}
V_0&=\{f\in L_2(\RR^d):\hat f(u)=F(z)\hat \varphi(u),\; F\in \cL(\ell_2)\},\\
V_j&=\{f\in L_2(\RR^d):\hat f(u)=\hat g(2^{-j}u),\; g\in V_0\},\quad j\in \ZZ\setminus\{0\}.
\end{align*}
In particular the mapping
$$
\Lambda:V_0\to \cL(\ell_2),\quad \Lambda(f)=F\;\text{ such that }\;\hat f(u)=F(z)\hat \varphi(u),$$
is an isomorphism of vector spaces. The Laurent series $F(z)$ is commonly called the \emph{mask} of the function $f$.

An important assumption on the scaling function $\varphi$ is the Riesz condition for its integer translates, equivalent to the
existence of constants $A,B>0$ such that
\begin{equation}\label{Riesz}
A\le\Phi(z)\le B,\quad z\in\TT^d,
\end{equation}
where
\begin{equation}\label{Phi}
\Phi(z)=\sum_{k\in\ZZ^d}\tilde\varphi(k)\,z^k=\sum_{k\in\ZZ^d}|\hat\varphi(u+2\pi k)|^2,
\end{equation}
with the \emph{autocorrelation function} 
$$
\tilde\varphi(u):=\int_{\RR^d}\varphi(t)\overline{\varphi(t-u)}\,dt,\quad\text{ that is,}\quad 
\tilde\varphi=\varphi*\overline{\varphi(-\cdot)}.$$
If in addition 
\begin{equation}\label{conds}
\varphi\in L_1(\RR^d),\quad H,\Phi\in \cL(\ell_1),
\end{equation}
then the usual axioms of a multiresolution analysis are satisfied, see e.g.~\cite{Stockler93}. From now on we assume that
\eqref{Riesz} and  \eqref{conds} are satisfied. In particular, $H,\Phi$ are continuous functions on $\TT^d$.

Refinement equation \eqref{refeqH} implies that $V_{j}\subset V_{j+1}$, $j\in\ZZ$, and the wavelet spaces $W_j$ are defined as
the orthogonal complements of $V_{j-1}$ in $V_{j}$, such that $L_2(\RR^d)=\bigoplus_{j\in\ZZ}W_j$ with pairwise orthogonal
spaces $W_j$. Since
$$
W_j=\{f\in L_2(\RR^d):\hat f(u)=\hat g(2^{-j}u),\; g\in W_0\},\quad j\in \ZZ\setminus\{0\},$$
any Riesz basis for a single space $W_j$ gives rise to Riesz bases for all of them by scaling, and hence to a Riesz basis for $L_2(\RR^d)$. 
Therefore  we look for a Riesz basis for the orthogonal complement $W_{0}$ of $V_{-1}$ in $V_0$.

In particular, such a Riesz basis may be generated by even translates of a finite set of functions called semi-orthogonal
wavelets or prewavelets.  That is, 
functions $\psi_s\in W_0$, $s\in S$, where $S$ is a finite index set, are said to be \emph{prewavelets} if the set
\begin{equation}\label{prewbasis}
\big\{\psi_s(\cdot-2k):s\in S,\;k\in\ZZ^d\big\}
\end{equation}
is a Riesz basis for $W_0$. 
Prewavelets are uniquely determined by their masks  $H_s=\Lambda(\psi_s)$. The coefficients of the Laurent series $H_s$ are
the coefficients of the reconstruction filter for the wavelet-based signal processing. The subset of $\ZZ^d$ for which these
coefficients are nonzero is called the \emph{support} of the mask $H_s$. A favorable feature for applications is when the
reconstruction filter  is finite, that is, $H_s$ are Laurent polynomials in $\cL$. Moreover, a reasonable task is to seek
filters with as  small mask supports as possible.

The rest of the paper is organized as follows. 
In Section~\ref{prewalg} we review the algebraic formulation of the conditions on multivariate prewavelets as a basis 
extension in  $\cL$ considered as a module of absolutely convergent Laurent series.
Section~\ref{prewBox} recalls the basic facts about multiresolution analysis generated by box splines. 
Section~\ref{Laur_ext} is devoted to new method of basis construction for the kernel of a linear form on 
the module $\cL$, while Section~\ref{multilevel} presents a related multilevel module construction. 
Finally, in Section~\ref{Section 7} we apply our algorithms to generate compactly supported prewavelets with significanly
fewer nonzero filter coefficients than those known in the literature for several  important box splines in 2D, as well as
a construction of  prewavelets for the piecewise linear box spline in 3D with a total of 155 nonzero filter
coefficients.

\section{Algebraic formulation of multivariate prewavelets}
\label{prewalg}

Under the condition that
 $H\in\cL(\ell_1)$, it follows from \eqref{refeqH} that
$$
\Lambda(V_{-1})=\cL_{\downarrow}(\ell_2)H,\quad\text{ with}\quad\cL_{\downarrow}(\ell_p):=\big\{F(z^2):F\in \cL(\ell_p)\big\}.$$
Then the mapping
$$
\Lambda_{\downarrow}:V_{-1}\to \cL_{\downarrow}(\ell_2),\quad \Lambda_{\downarrow}(f)=F\;\text{ such that }\;\hat f(u)=F(z)H(z)\hat \varphi(u),$$
is also an isomorphism of vector spaces.

For all $f,g\in V_0$, see e.g.~\cite{JM,CSW92}, we have
\begin{align}\label{inprodcon}
\begin{split}
\langle f,g\rangle_{L_2(\RR^d)}&=(2\pi)^d\int_{\RR^d}\hat f \overline{\hat g}\,dx
=(2\pi)^d\int_{\TT^d}\Phi\,\Lambda(f)\overline{\Lambda(g)}\,du\\
&=(2\pi)^d\int_{[0,\pi]^d}\langle \Lambda(f),\Lambda(g)\rangle_\Phi\,du,
\end{split}
\end{align}
where
\begin{equation}\label{inprod}
\langle F,G\rangle_\Phi(z):=\sum_{s\in E} \Phi\big((-1)^sz\big)F\big((-1)^sz\big)\overline{G\big((-1)^sz\big)},
\quad F,G\in \cL(\ell_2),
\end{equation}
$E=\{0,1\}^d$ is the set of vertices of the $d$-dimensional unit cube $[0,1]^d$, and 
$$
(-1)^s z:=((-1)^{s_1} z_1,\ldots,(-1)^{s_d}z_d),\qquad s\in\ZZ^d,\quad z\in\CC^d.$$

It is easy to see that for every $T\in \cL(\ell_2)$ the condition 
$$
T\big((-1)^sz\big)=T(z)\quad \text{for all}\quad s\in E$$
is equivalent to $T\in \cL_{\downarrow}(\ell_2)$. Given any pair  $F,G\in \cL(\ell_2)$, the function  
$\langle F,G\rangle_\Phi$ 
belongs to $L_1(\TT^d)$ 
because the product $\Phi F\overline{G}$  of a continuous function $\Phi$ and two square
integrable functions  $F,G$ is absolutely integrable on $\TT^d$. The function $\langle F,G\rangle_\Phi$ 
is $\pi$-periodic with respect to each component of $u$ and gives rise to a Fourier series with zero coefficients $a_k$ for 
$k\notin 2\ZZ^d$ that however does not have to converge.
If one of $F,G$ belongs to $\cL(\ell_1)$ and the other to
$\cL(\ell_2)$, then $\langle F,G\rangle_\Phi\in \cL_{\downarrow}(\ell_2)$. Moreover,
\begin{equation}\label{inprodval}
\langle F,G\rangle_\Phi\in \cL_{\downarrow}(\ell_1)\quad\text{for all}\quad F,G\in \cL(\ell_1).
\end{equation}
In addition, it follows that 
\begin{equation}\label{bilin}
\langle TF,G\rangle_\Phi=T\langle F,G\rangle_\Phi,
\end{equation}
as long as $F,G\in \cL(\ell_2)$, $T\in \cL_{\downarrow}(\ell_2)$ and at least one of $F,G,T$ belongs to $\cL(\ell_1)$.

In particular, since $H\in\cL(\ell_1)$, we obtain from \eqref{inprodcon},
$$
\langle f,g\rangle_{L_2(\RR^d)}=(2\pi)^d\int_{[0,\pi]^d}\Lambda_{\downarrow}(f)\,\big\langle H,\Lambda(g)\big\rangle_\Phi\,du,
\qquad f\in V_{-1},\quad g\in V_0,$$
and hence 
\begin{equation}\label{orthog}
g\in W_0\quad\Longleftrightarrow\quad \big\langle H,\Lambda(g)\big\rangle_\Phi=0.
\end{equation}

The Laurent series ring $\cL(\ell_1)$ is a module over its subring $\cL_{\downarrow}(\ell_1)$, and
the Laurent polynomial ring $\cL$ is a module over 
$$
\cL_{\downarrow}:=\big\{P(z^2):P\in \cL\big\}.$$
Note that $\cL(\ell_2)$ can be considered as a module over $\cL(\ell_1)$ or $\cL_{\downarrow}(\ell_1)$ 
because the product of a function in $L_2(\TT^d)$ with a continuous function on $\TT^d$ belongs to $L_2(\TT^d)$. 

In addition to \eqref{inprodval} and \eqref{bilin}, the form 
$\langle \cdot,\cdot\rangle_\Phi:\cL^M(\ell_1)\times \cL^M(\ell_1)\to \cL_{\downarrow}(\ell_1)$ satisfies 
$$
\langle F,G\rangle_\Phi=\overline{\langle G,F\rangle_\Phi}\quad\text{and}\quad\langle F,F\rangle_\Phi(z)\ge0,\quad z\in\TT^d,
$$
since $\Phi(z)$ is real and positive on $\TT^d$, which resembles the properties of an inner product.
We also note the following useful identity obtained by a standard calculation, see  for example \cite[Eq.~(2.2)]{CSW92},
\begin{equation} \label{Phiz2}
\langle H,H\rangle_\Phi(z)=\Phi(z^2).
\end{equation}

\begin{theorem}\label{charprew} 
Assume that $\varphi,H,\Phi$ satisfy \eqref{Riesz} and  \eqref{conds}.
Let $\psi_s\in V_0$, $s\in S$, be such that $H_s:=\Lambda(\psi_s)\in \cL(\ell_1)$. 
Then $\psi_s$ are prewavelets if and only if $|S|=2^d-1$,
\begin{equation}\label{modbasis}
\{H\}\cup\{H_s:s\in S\} \text{ is a basis for the  $\cL_{\downarrow}(\ell_1)$-module $\cL(\ell_1)$,}
\end{equation}
and 
\begin{equation}\label{orthogHs}
\big\langle H,H_s\big\rangle_\Phi=0,\quad s\in S.
\end{equation}
\end{theorem}

\begin{proof} Let $\varphi_r(x):=\varphi(x-r)$, $r\in E$. Then 
$\big\{\varphi_r(\cdot-2k):r\in E,\;k\in\ZZ^d\big\}$ is a Riesz basis for $V_0$ because it coincides with the set of all
integer translates of $\varphi$. The functions $\Lambda(\varphi_r)(z)=z^r$, $r\in E$, obviously form a basis for the 
$\cL_{\downarrow}(\ell_1)$-module $\cL(\ell_1)$. Hence there exist unique $H_{sr}\in \cL_{\downarrow}(\ell_1)$ such that
$$
H_s(z)=\sum_{r\in E}H_{sr}(z)z^r,\qquad  s\in\{0\}\cup S,$$
where $H_0:=H$. 

Since $W_0$ is the orthogonal complement of $V_{-1}$ in $V_0$ and in view of \eqref{orthog}, the set
\eqref{prewbasis} is a Riesz basis for $W_0$ if and only if \eqref{orthogHs} holds and the
 translates $\varphi(\cdot-2k)$, $\psi_s(\cdot-2k)$, $s\in S$, $k\in\ZZ^d$,
form a Riesz basis for $V_0$. According to
\cite[Theorem 4.3]{JM}, the latter conditon is satisfied if and only if the matrix
$$
\cA(z):=[H_{sr}(z)]_{s\in \{0\}\cup S,\;r\in E}$$
is square (that is, $|S|=2^d-1$) and nonsingular for all $z\in\TT^d$. In other words, 
$\det \cA(z)\in \cL_{\downarrow}(\ell_1)$
does not vanish in $\TT^d$. By Wiener's lemma  this means that 
$\det \cA(z)$ is invertible in the ring $\cL_{\downarrow}(\ell_1)$, which in turn is equivalent to \eqref{modbasis}.
\end{proof}

Assume that $\{H\}\cup\{P_s:s\in S\}$, with $|S|=2^d-1$, is a basis for the $\cL_{\downarrow}(\ell_1)$-module $\cL(\ell_1)$.
Then $H_s\in \cL(\ell_1)$ satisfying \eqref{modbasis} and \eqref{orthogHs} may be found by the orthogonalization formula 
\begin{equation}\label{JMorthog}
H_s=\big\langle P_s,H\big\rangle_\Phi H-\big\langle H,H\big\rangle_\Phi P_s
\end{equation}
as suggested in \cite[Section 6]{JM}. Indeed, in this case $\big\langle H,H_s\big\rangle_\Phi=0$ follows from \eqref{bilin}
because $\big\langle P_s,H\big\rangle_\Phi$ and $\big\langle H,H\big\rangle_\Phi$ belong to $\cL_\downarrow(\ell_1)$.
This gives rise to prewavelets $\psi_s$ determined by the generators $\Lambda(\psi_s)=H_s$.
Note that the existence of $P_s$, $s\in S$, with the above basis property is equivalent to the \emph{extensibility} of the
$2^d$-tuple $[H_{0r}]_{r\in E}$ in the sense of \cite{JM}, with the resulting nonsingular matrix of $P_{sr}\in \cL_{\downarrow}(\ell_1)$,
$s\in \{0\}\cup S$, $r\in E$, determined by the expansions $P_s(z)=\sum_{r\in E}P_{sr}(z)z^r$. 

Furthermore, by a Gram-Schmidt orthogonalization procedure as in the proof of \cite[Theorem~4.4]{JM}
we may obtain $H_s$ satisfying $\big\langle H_s,H_r\big\rangle_\Phi=0$ for all $r\ne s$, which in view of \eqref{inprodcon} 
means that the subspaces of $W_0$ generated by the translates of $\psi_s$ are mutually orthogonal. However, this does not
guarantee that the translates in \eqref{prewbasis} are orthogonal between them. 

Scaling functions and prewavelets with compact support are particularly useful for applications, especially when $H$ and 
$H_s$ are Laurent polynomials, since in this case only finitely many nonzero coefficients are needed for the 
wavelet reconstruction transformation. We formulate here 
a technical result needed in the proof of the existence of 
compactly supported  prewavelets \cite{JM} based on the Quillen-Suslin theorem, see Section~\ref{Laur_ext}.

\begin{lemma}\label{linindep}
Let the scaling function $\varphi\in L_2(\RR^d)$ be compactly supported. Assume that the integer translates 
of $\varphi$ are linearly independent, that is, for each finite $M\subset \ZZ^d$ the set of functions
$\varphi(\cdot-k)$, $k\in M$, is  linearly independent. Then $H\in\cL$, and the polynomials $H_{0r}\in\cL_\downarrow$ in the
expansion $H(z)=\sum_{r\in E}H_{0r}(z)z^r$ do not have common zeros in $(\CC\setminus\{0\})^d$.
\end{lemma}
\noindent 
For a proof, see \cite[Theorem~8.1]{JM}.

\section{Box spline prewavelets}
\label{prewBox} 

Let $\Xi=[\xi_1,\ldots,\xi_\ell]\in\ZZ^{d\times \ell}$, $\ell\ge d$, be a full rank matrix with nonzero columns $\xi_j$.
The box spline $M_\Xi$ associated with $\Xi$ is a non-negative function on $\RR^d$ that may be defined by its Fourier
transform
\cite{BHR93}
$$
\widehat M_\Xi(u):=\prod_{j=1}^\ell\frac{1-e^{-i\xi_{j}^Tu}}{i \xi_{j}^T\!u},
\qquad u\in\RR^d. $$%
Hence the refinement equation \eqref{refeqH} is satisfied for $\varphi=M_\Xi$ with
\begin{equation}\label{Hbox}
H(z)=2^{-\ell}\prod_{j=1}^\ell(1+z^{\xi_{j}}).
\end{equation}

It is easy to check that $\tilde\varphi =M_\Xi*M_\Xi(-\cdot)$ satisfies
$$
\tilde\varphi(x)=M_{\Xi\cup \Xi}(x+c_\Xi),\qquad c_\Xi=\sum_{j=1}^\ell \xi_j,$$
and hence $\tilde\varphi$ coincides with the \emph{centered box spline} $M_{\Xi\cup \Xi}^c$,
see \cite[p.~10--11]{BHR93}. Thus, the function $\Phi$ of \eqref{Phi} is given by
\begin{equation}\label{Phibox}
\Phi(z)=\sum_{k\in\ZZ^d}M_{\Xi\cup \Xi}^c(k)\,z^k.
\end{equation}

According to \cite{DM83}, the Riesz condition \eqref{Riesz} holds and hence $\varphi=M_\Xi$ generates a multiresolution
analysis if and only if the matrix $\Xi$ is \emph{unimodular}, that is $|\det X|=1$ for every non-singular $d\times d$
submatrix $X$ of $\Xi$. Moreover, the integer translates of $M_\Xi$  are linearly independent if and only if
 $\Xi$ is unimodular.

Notice that both $H(z)$ and $\Phi(z)$ are polynomials with positive coefficients. 

In what follows we will look for Laurent polynomials $H_s\in \cL$, $s\in E\setminus\{0\}$,  satisfying \eqref{modbasis} and 
\eqref{orthogHs}, such that the prewavelets $\psi_s$ given by $\Lambda(\psi_s)=H_s$  are compactly supported. 

If the polynomials $H$ and $\Phi$ are invariant with respect to arbitrary permutations of the variables $z_1,\ldots,z_d$,
then polynomials $H_s$ will split into subsets generated from a single polynomial by all possible permutations
of the variables, which is a useful property for applications. We now identify box splines satisfying this property.

\begin{lemma}\label{permbox} 
Suppose that for every permutation of columns of $\Xi$ there is a permutation of its rows such that  $\Xi$ is unchanged after
applying both permutations. Then both  $H$ and $\Phi$ are invariant with respect to the permutations of the variables 
$z_1,\ldots,z_d$.
\end{lemma}

\begin{proof}
This follows immediately from \eqref{Hbox} and \eqref{Phibox}.
\end{proof}

\section{Basis extension in a Laurent polynomial module}\label{Laur_ext}

If the scaling function $\varphi\in L_2(\RR^d)$ is compactly supported and its integer translates are linearly independent,
then \eqref{conds} is satisfied, and both $H$ and $\Phi$ are Laurent polynomials in $\cL$. 
Condition \eqref{Riesz} is also satisfied in this case, see \cite[Section 5]{JM}.

As shown in \cite{JM}, under these conditions the existence of an extension of $H$ to a basis  $\{H\}\cup\{P_s:s\in S\}$
for $\cL(\ell_1)$ over $\cL_{\downarrow}(\ell_1)$, where $P_s$ are polynomials in $\cL$, can be proved by using the
Quillen-Suslin theorem. We now sketch this proof in matrix-free terms.  
Using Lemma~\ref{linindep} and Hilbert's Nullstellensatz it is easy to show that
$\{H_{0r}:r\in E\}$ is a \emph{unimodular row} in $\cL_{\downarrow}$, that is, there exist 
$Q_r\in \cL_{\downarrow}$, $r\in E$, such that $\sum_{r\in E}H_{0r}Q_r=1$. 
It follows that the cyclic submodule $\cL_\downarrow H$ is a direct summand of $\cL$, namely 
$$
\cL=\cL_\downarrow H\oplus\cP,\quad\text{where}\quad
\cP:=\big\{P\in\cL:\sum_{r\in E}P_{r}Q_r=0,\;\text{with}\; P(z)=\sum_{r\in E}P_{r}(z)z^r,\; P_{r}\in\cL_\downarrow\big\}.$$
As a direct summand of a free module $\cL$, the submodule $\cP$ is projective, and since $\cL$ is finitely generated, so is
$\cP$. Hence, by the Quillen-Suslin theorem, $\cP$ is free. Clearly, its rank is $|E|-1=2^d-1$. Therefore that exists a basis
$P_s$, $s\in S$, with $|S|=2^d-1$, that gives the required extension of $H$.

Hence, basis polynomials $H_s\in\cL$ satisfying \eqref{modbasis} and \eqref{orthogHs} can be computed by applying a
constructive version of the Quillen-Suslin theorem (see e.g.~\cite{Park04}) and orthogonalizing according to \eqref{JMorthog}. 

Alternatively, for box splines in dimensions at most 3, an algorithm for the generation of the basis polynomials
$H_s$ has been suggested in \cite{RS}, and modified in \cite{BDG} to reduce the number of nonzero coefficients 
of $H_s$, which makes resulting prewavelets more efficient in application.

We suggest below an alternative method that leads to even fewer nonzero coefficients of $H_s$ for the particular box spline
prewavelets considered in the literature.

\subsection{Bases for the kernel of a linear form} %

Since $\cL_{\downarrow}$ is a subring of $\cL$ and  $\cL_{\downarrow}(\ell_1)$ is a subring of $\cL(\ell_1)$, $\cL$ may be
considered as an ${\cL_{\downarrow}}$-module and $\cL({\ell_1})$ may be considered as an
${\cL_{\downarrow}}({\ell_1})$-module. It is well known that
\begin{equation}
  \cL (\ell_1)=\oplus_{j \in E}\cL_{\downarrow}(\ell_1)z^j, \label{Eq 2.1}
\end{equation}
where $E = {\left\{ {0,\,1} \right\}^d}$. So, each element $A$ of $\cL$ (of $\cL(\ell_1)$) may be uniquely presented in the
form $A(z) = \sum\nolimits_{j \in E} {{a_j}{z^j}} $, where ${a_j} \in {\cL_{\downarrow}}$ (${a_j} \in
{\cL_{\downarrow}}({\ell_1})$). Then the mapping $\rho :A \mapsto {a_0}$ is linear over ${\cL_{\downarrow}}(\ell_1)$
and maps $\cL(\ell_1)$ onto ${\cL_{\downarrow}}({\ell_1})$. The restriction of $\rho :A(z) \mapsto {a_0}$ to $\cL$ is linear
over ${\cL_{\downarrow}}$ and maps $\cL$ onto ${\cL_{\downarrow}}$. Note that
\begin{equation}\label{rho}
\rho (A(z))=2^{-d}\sum_{s\in E} A\big((-1)^sz\big).
\end{equation}

Let $D(z)$ be a fixed polynomial from $\cL$. Let ${\rho_D} $ and $\tau_D $  
be  mappings of  $ \cL({\ell_1})  $  onto $ {\cL_{\downarrow}}({\ell_1}) $  given by 
$${\rho_D}:F(z) \mapsto \rho (D(z)  \overline {F(z)} )\quad \text{and} \quad
{\tau_D}:F(z) \mapsto \rho (\overline{D(z)}  {F(z)} ) $$ 
for any element $F(z)$ of $\cL({\ell_1})$.
Since the conjugation \eqref{conj} is an automorphism of $\cL({\ell_1})$, it is not difficult to show that 
${\rho_D} $ and $\tau_D $  are linear forms on the ${\cL_{\downarrow}}({\ell_1})$-module  $ \cL({\ell_1})$.

\begin{lemma}\label{l2.1}
In the above notations, $\Ker {\rho_D} = \Ker {\tau_D} $. 
\end{lemma}
\begin{proof} 
It is not difficult to note that $ \overline{ {\rho_D}(F(z))} =  
\overline{ {\rho}(D(z)\overline{F(z)})} = {\rho}(\overline{ D(z)\overline{ F(z)}}) = 
 {\rho}(\overline{ D(z)}F(z)) = {\tau_D}(F(z))$. Therefore, ${\rho_D}(F(z)) = 0$  
 if and only if $ {\tau_D}(F(z)) = 0$ and the assertion follows. 

\end{proof}

\begin{lemma}\label{l2.2}
Suppose that polynomials $\{ {H_j}(z)|j \in E\setminus \{ 0\} \}  \subseteq \cL$ form a basis of $\Ker{\rho_D}$ over ${\cL_{\downarrow}}({\ell_1})$, and suppose that there is a polynomial $H(z) \in \cL$ such that ${\rho_D}(H(z))$ is invertible in ${\cL_{\downarrow}}({\ell_1})$. Then polynomials $\{ {H_j}(z)|j \in E\setminus \{ 0\} \} $ and $H(z)$ form a basis of $\cL({\ell_1})$ over ${\cL_{\downarrow}}({\ell_1})$. 
\end{lemma}

\begin{proof} 
It follows from the definition of ${\rho_D}$ that if ${\rho_D}(H(z)   F(z)) = $ ${\rho_D}(H(z))   F(z) = 0$ for some
nonzero $F(z) \in {\cL_{\downarrow}}({\ell_1})$ then ${\rho_D}(H(z)) = 0$. But ${\rho_D}(H(z)) \ne 0$ because
${\rho_D}(H(z))$ is invertible in ${\cL_{\downarrow}}({\ell_1})$ and hence ${\rho_D}(H(z))   F(z) \ne 0$ for all nonzero
$F(z) \in {\cL_{\downarrow}}({\ell_1})$. It implies that  $H(z)   {\cL_{\downarrow}}({\ell_1}) \cap \Ker{\rho_D} = 0$ and
hence the sum $H(z)   {\cL_{\downarrow}}({\ell_1}) + \Ker{\rho_D} = $ $H(z)   {\cL_{\downarrow}}({\ell_1}) \oplus
\Ker{\rho_D}$ is direct. Thus, it is sufficient to show that $H(z)   {\cL_{\downarrow}}({\ell_1}) + \Ker{\rho_D} =
\cL({\ell_1})$. Let $E(z) \in \cL({\ell_1})$, since the element ${\rho_D}(H(z))$ is invertible in
${\cL_{\downarrow}}({\ell_1})$, there is an element $F(z) \in {\cL_{\downarrow}}({\ell_1})$ such that ${\rho_D}(H(z))   F(z)
= {\rho_D}(E(z))$. Therefore, ${\rho_D}(H(z)   F(z) - E(z)) = 0$ and hence $E(z) \in H(z)   {\cL_{\downarrow}}({\ell_1}) +
\Ker{\rho_D}$. So, we can conclude that $H(z)   {\cL_{\downarrow}}({\ell_1}) + \Ker{\rho_D} = \cL({\ell_1})$. 
\end{proof}

Let $ k \in {\mathbb{Z}}^d $, as the element $ z^k $ is invertible in $ \cL $, the 
multiplication by $ z^k $ induces an automorphism of the $\cL_{\downarrow}(\ell_1)$-module 
$ \cL(\ell_1)$ which permutes the direct summands  of  $\cL(\ell_1) =\oplus_{j \in E}\cL_{\downarrow}(\ell_1)z^j $.

\begin{proposition}\label{p2.1}
 Let $D(z) = \sum\nolimits_{j \in E} {{d_j}{z^{k_j}}}  \in \cL$, where ${d_j} \in {\cL_{\downarrow}}$, ${d_j}{z^{k_j}}  \in
 {\cL_{\downarrow}}{z^j}$ and ${k_j} \in {\mathbb{Z}}^d$.  Suppose that ${d_0} \ne 0$ is invertible in
 ${\cL_{\downarrow}}({\ell_1})$ (i.e. ${d_0} \ne 0$ has no zeros on the unit torus). Let ${D_j}(z) = {d_j} - {d_0}{z^{ -k_
 j}}$ and ${H_j}(z) = \overline {{D_j}(z)} = {\bar d_j} - {\bar d_0}{z^{k_j}} $, where $j \in E\setminus \{ 0\} $. Then:
  \begin{romanlist}[(i)]
\item	$\rho (D(z) {D_j}(z)) = \rho (\overline{D(z)} {H_j}(z)) = 0$ for any 
$j \in E\setminus \{ 0\}$;
\item	polynomials ${H_j}(z)$ are linearly independent over ${\cL_{\downarrow}}({\ell_1})$, 
where $j \in E\setminus \{ 0\} $; 
\item	elements ${H_j}(z)$ form a basis of $\Ker{\rho_D}$ over ${\cL_{\downarrow}}({\ell_1})$, 
where $j \in E\setminus \{ 0\} $;
\item	if $H(z) \in \cL$ is a polynomial  such that ${\rho_D}(H(z))$ is invertible in ${\cL_{\downarrow}}({\ell_1})$ then
polynomials $\{ {H_j}(z)|j \in E\setminus \{ 0\} \} $ and $H(z)$ form a basis of $\cL({\ell_1})$ over
${\cL_{\downarrow}}({\ell_1})$. \end{romanlist}
\end{proposition}

\begin{proof} 
Since ${d_j}{z^{k_j}}  \in {\cL_{\downarrow}}{z^j}$ and ${d_j} \in {\cL_{\downarrow}}$, it is not difficult to show that
\begin{equation}
 {z^{k_j}} = {z^{r_j}}{z^{j}}  \label{Eq 2.2}
 \end{equation}
 for any ${j \in E}$,  where $ {r_j} \in ({2\mathbb{Z}})^d $. It easily implies that   
${\cL_{\downarrow}(\ell_1)}{z^j} = {\cL_{\downarrow}(\ell_1)}{z^{k_j}}$. Then it follows from (\ref{Eq 2.1}) that
\begin{equation}
  \cL (\ell_1)=\oplus_{j \in E}\cL_{\downarrow}(\ell_1)z^{k_j}. \label{Eq 2.3}
\end{equation}
(i) Since the multiplication by $ z^k $  permutes the direct summands  of $\cL = 
\oplus_{j \in E}\cL_{\downarrow}z^j $, it is not difficult to note  that $\rho (D(z) z^{-k_j}) = d_j$ for any 
$ {j \in E} $. Then, as the mapping $\rho$ is linear over  $ \cL_{\downarrow} $, we have $\rho (D(z) {D_j}(z)) = 
\rho (D(z) ({d_j} - {d_0}{z^{ -k_ j}}))  =  {d_j}{ \rho (D(z))} - {d_0}{\rho (D(z){z^{ -k_ j}})} = 
{d_j}{ d_0} - {d_0}{d_j} = 0$ for any $j \in E\setminus \{ 0\}$.  Therefore, 
$0 =  \overline{\rho (D(z) {D_j}(z))} = \rho (\overline{D(z) {D_j}(z)}) = 
\rho (\overline{D(z)} {H_j}(z))$.

(ii)  Consider a linear combination 
$$\sum_{j \in E\setminus \{ 0\} } {\lambda _j}{H_j}(z) = 
\sum_{j \in E\setminus \{ 0\} } {\lambda _j}({\bar d_j} - {\bar d_0}{z^{k_j}})  =
 ( \sum_{j \in E\setminus \{ 0\} } {\lambda _j}{\bar d_j}) - 
 (\sum_{j \in E\setminus \{ 0\} } 
 {\lambda_j }{\bar d_0}z^{k_j}), $$  	
 where ${\lambda_j} \in {\cL_{\downarrow}} (\ell_1)$. 
Since  	${\lambda_j} \in {\cL_{\downarrow}}({\ell_1})$, we see that 
$(\sum_{j \in E\setminus \{ 0\} } {\lambda _j}
{\bar d_j}) \in {\cL_{\downarrow}}({\ell_1}) $. Therefore, by (\ref{Eq 2.3}),  if $\sum_{j \in E\setminus \{ 0\} }
 {\lambda _j}{H_j}(z) =  0$ then $ {\lambda _j}{{\bar d}_0}
 {z^{k_j}} = 0 $ for any $j \in E\setminus \{ 0\}$. But, as ${d_0} \ne 0$, 
 it is possible if and only if all ${\lambda _j} = 0$. 
 Thus, ${H_j}(z) = {\bar d_j} - {\bar d_0}{z^j}$, where $j \in E\setminus \{ 0\}$, 
 are linearly independent over ${\cL_{\downarrow}}({\ell_1})$.   

(iii) It follows from (i) and Lemma \ref{l2.1} that all ${H_j}(z) \in \Ker{\rho_D} =  
\Ker{\tau_D}$. Since ${d_0} \ne 0$ has no zero divisors 
in ${\cL_{\downarrow}}({\ell_1})$, 	it is also not difficult to note that $\Ker{\rho_D} \cap {\cL_{\downarrow}}({\ell_1})
 = 0$. 	Let $A(z) = 	 \sum\nolimits_{j \in E} {{a_j}{z^j}}  \in \Ker{\rho_D}$, 
 where ${a_j} \in {\cL_{\downarrow}}({\ell_1})$ 	 and $E = {\left\{ {0,\,1} \right\}^d}$.  
As ${H_j}(z) = {\bar d_j} - {\bar d_0}{z^{k_j}}$, it follows from (\ref{Eq 2.2}) that 
${H_j}(z) = {\bar d_j} - ({\bar d_0}{z^{r_j}}){z^{j}}$, where $ {r_j} \in ({2\mathbb{Z}})^d$. 
Since the element $ d_0 $ is invertible in $ \cL_{\downarrow}(\ell_1) $, it is easy to note  that
there are elements ${\lambda _j} =  - ({\bar d_0}{z^{r_j}})^{ - 1}{a_j} \in  \cL_{\downarrow}(\ell_1) $. 
Then it is not difficult to show that $A(z) - \sum_{j \in E\setminus \{ 0\} } {\lambda _j}
{H_j}(z) = a  \in \Ker{\rho_D} \cap {\cL_{\downarrow}}({\ell_1}) = 0$. Therefore,  
$A(z) = \sum_{j \in E\setminus \{ 0\}} {{\lambda _j}{H_j}(z)} $ and the assertion follows from (ii). 

	(iv). The assertion follows from (iii) and Lemma \ref{l2.2}. 
\end{proof}

\begin{proposition}\label{p2.2}
Let $U(z)$ be a polynomial of $\cL$ and let ${\rho_U}:\cL({\ell_1}) \mapsto {\cL_{\downarrow}}({\ell_1})$ be a mapping given
by  ${\rho_U}:F(z) \mapsto \rho (U(z)   \overline {F(z)} )$ for any element $F(z)$ of $\cL({\ell_1})$. Suppose that there is
an invertible in $\cL({\ell_1})$ polynomial $V(z) \in \cL$ such that ${d_0} \ne 0$ is invertible in
${\cL_{\downarrow}}({\ell_1})$, where $D(z) = \sum_{j \in E} {{d_j}{z^j}} = U(z)   V(z)$. Let ${D_j}(z) = {d_j} - {d_0}{z^{ -
j}}$, where $j \in E\setminus \{ 0\} $. Then:
 \begin{romanlist}[(i)]
\item	$\rho (U(z)   V(z)   {D_j}(z)\,) = 0$;
\item  elements ${H_j} = \overline{V(z)}  \overline{{D_j}(z)} $ form a basis of $\Ker{\rho_U}$ over 
${\cL_{\downarrow}}({\ell_1})$, where $j \in E\setminus \{ 0\} $. 
\item	if $H(z) \in \cL$ is a polynomial  such that ${\rho_U}(H(z))$ is invertible in ${\cL_{\downarrow}}({\ell_1})$ then
polynomials $\{ {H_j}(z)|j \in E\setminus \{ 0\} \} $ and $H(z)$ form a basis of $\cL({\ell_1})$ over
${\cL_{\downarrow}}({\ell_1})$. 
 \end{romanlist}
\end{proposition}

\begin{proof}
(i) We can apply Proposition \ref{p2.1}(i) to the case where $D(z) = \sum\nolimits_{j \in E} {{d_j}{z^j}} $ =
$U(z)  V(z)$.

(ii) By Proposition \ref{p2.1}(iii), elements $\overline {{D_j}(z)} $ form a basis of $\Ker{\rho_D}$ over
${\cL_{\downarrow}}({\ell_1})$, where $j \in E\setminus \{ 0\} $ and ${\rho_D}$ is a mapping ${\rho_D}:\cL({\ell_1})
\mapsto {\cL_{\downarrow}}({\ell_1})$ which is given by ${\rho_D}:A(z) \mapsto \rho (D(z)   \overline {A(z)} )$ for any
element $A(z)$ of $\cL({\ell_1})$. As the polynomial $V(z) \in \cL$ is invertible in $\cL({\ell_1})$, the polynomial
$\overline {V(z)} $ is also invertible in $\cL({\ell_1})$. Therefore, the multiplication by $\overline {V(z)} $ induces an
${\cL_{\downarrow}}({\ell_1})$-module automorphism on $\cL({\ell_1})$ which maps $\Ker{\rho_D}$ on $\Ker{\rho_U}$. Then we
can conclude that  polynomials ${H_j} = \overline {V(z) {D_j}(z)} $  form a basis of $\Ker{\rho_U}$ over
${\cL_{\downarrow}}({\ell_1})$, where $j \in E\setminus \{ 0\} $. 

(iii)  The assertion follows from (ii) and Lemma \ref{l2.2}. 
\end{proof}

Let $i =  (i_1, i_2,...,i_d) \in  {\mathbb{Z}}^d$ then we put $| i | =  ( | i_1 | ,  | i_2 | ,..., | i_d | )$ .

\begin{corollary}\label{c2.1} 
Let $U(z) = \sum\nolimits_{j \in E} {{u_j}{z^j}}  \in \cL$. Suppose that ${u_i} \ne 0$ is invertible in
${\cL_{\downarrow}}({\ell_1})$ for some $i \in E\setminus \{ 0\} $. Then: 
\begin{romanlist}[(i)]
\item elements ${H_{|j-i|}} = \overline{u_j} {z^{ i}} - \overline{u_i}{z^{j}} $ form a basis of $\Ker{\rho_U}$ over 
${\cL_{\downarrow}}({\ell_1})$, where $j \in E\setminus \{ i\} $; 
\item if $H(z) \in \cL$ is a polynomial  such that ${\rho_U}(H(z))$ is invertible in ${\cL_{\downarrow}}({\ell_1})$ then
polynomials $\{ {H_j}(z)|j \in E\setminus \{ 0\} \} $ and $H(z)$ form a basis of $\cL({\ell_1})$ over
${\cL_{\downarrow}}({\ell_1})$. 
\end{romanlist}
\end{corollary}
\begin{proof}
We can apply Proposition \ref{p2.2} in the case where $V(z) = {z^{ - i}}$. 

(i) Let $D(z) = U(z){z^{ - i}} = \sum_{j \in E} {{d_j}{z^j}} $ and let ${D_j}(z) = {d_j} - {d_0}{z^{ - j}}$, where $j \in E\setminus \{ 0\} $. Then   $D(z) =  \sum_{j \in E} {d_{|j-i|}}{z^{(j-i)}}$, where $ {d_{|j-i|}} = u_j$ and hence $D_{|j-i|}(z) = {u_j} - {u_i}z^{ -( j-i)}$, where $j \in E\setminus \{i\} $. 
Therefore, by Proposition \ref{p2.2}(i),  ${H_{|j-i|}} = \overline {{z^{ - i}}  {D_{|j-i|}}(z)} $ 
form a basis of $\Ker{\rho_U}$ over  ${\cL_{\downarrow}}({\ell_1})$, where $j \in E\setminus \{ i\} $. 
 Since ${H_{|j-i|}} = \overline {{z^{ - i}}  {D_{|j-i|}}(z)} =  \overline{z^{ - i}}  
 (\overline{{u_j} - {u_i}{z^{ -( j-i)}}})= z^{ i} (\overline{u_j} -  \overline{u_i}z^{ ( j-i)}) = 
  \overline{u_j}z^{ i} -  \overline{u_i}z^{ j}$, the assertion follows. 
    
  (ii) The assertion follows from Proposition 
  \ref{p2.2}(iii).
\end{proof}

\subsection{Bases for the orthogonal complement  with respect to a bilinear form}
\label{Section 3}

Put $U(z) =c H(z) \Phi(z)$, where $c\in  \mathbb{C} $ is an arbitrary nonzero constant that only affects 
the scaling of the resulting basis and can be chosen individually for particular $H$ and $\Phi$.
Then it follows from the definition of ${\left\langle {G,F} \right\rangle_\Phi}$ and \eqref{rho} that 
\begin{equation}   \label{eq3.1}
 {\left\langle {H,F} \right\rangle_\Phi} = {2^d}c^{-1}\rho (U(z)  \overline {F(z)})  = {2^d}c^{-1}{\rho_U}(F(z)).                      
\end{equation}
 Thus, ${\left\langle {H,F} \right\rangle _\Phi} = 0$ if and only if 
 $$\rho (U(z) \overline {F(z)} \,) = {\rho_U}(F(z)) = 0.$$                                        
 
  \begin{proposition}    \label{p3.1}                                 
 The polynomial $\Phi(z)$ has  the following  properties: 
\begin{romanlist}
\item $\overline{\Phi(z)} = \Phi(z)$ and the function $\Phi(z)$ has only real values; 
\item there is a positive real number $A$ such that   $\left| {\Phi(z)} \right| \ge A$  on the unit torus; 
\item  $ {\left\langle {H,H} \right\rangle_\Phi} = {2^d}c^{-1}\rho (U(z) \overline {H(z)}) = 
{2^d}c^{-1}{\rho_U}(H(z)) = \Phi({z^2})$.                         
\end{romanlist}
\end{proposition}

\begin{proof}
It is easy to see that  (i) follows from \eqref{Phi},
(ii) from \eqref{Riesz}, and (iii) from \eqref{Phiz2}.
\end{proof}

\begin{proposition}\label{p3.2} 
 Suppose that there exists an invertible in $\cL({\ell_1})$ polynomial $V(z) \in \cL$ such that ${d_0} \ne 0$ is invertible in ${\cL_{\downarrow}}({\ell_1})$, where $D(z) = cH(z)\Phi(z)V(z) = \sum\nolimits_{j \in E} {{d_j}{z^j}} $. Let ${D_j}(z) = {d_j} - {d_0}{z^{ - j}}$ and let ${H_j}(z) = \overline {V(z){D_j}(z)} $, where $j \in E\setminus \{ 0\} $. Then 
\begin{romanlist}[(i)]
\item  ${\left\langle {H,{H_j}(z)} \right\rangle _\Phi} = 0$ for all $j \in E\setminus \{ 0\} $ ;
\item polynomials $\{ {H_j}(z)|j \in E\setminus \{ 0\} \} $ and $H(z)$ form a basis of $\cL({\ell_1})$ over ${\cL_{\downarrow}}({\ell_1})$. 
\end{romanlist}
\end{proposition}
\begin{proof}
(i) Put $U(z) = cH(z)   \Phi(z)$ then it follows from \eqref{eq3.1} that ${\left\langle {H,F} \right\rangle _\Phi} = 0$ if
and only if $\rho (U(z) \overline {F(z)}) = {\rho_U}(F(z)) = 0$. Then the assertion follows from Proposition 
\ref{p2.2}(ii).

(ii) Since $\Phi({z^2}) = {\left\langle {H,H} \right\rangle _\Phi} = 2^d c^{-1}\rho (U(z) \overline {H(z)} \,) = 2^d
c^{-1}{\rho_U}(H(z))$ and there is a positive real number $a$ such that $\left| {\Phi(z)} \right| \ge a$ on the unit torus, we
can conclude that ${\rho_U}(H(z))$ has no zeros on the unit torus. Therefore, ${\rho_U}(H(z))$ is invertible in
${\cL_{\downarrow}}({\ell_1})$ and the assertion follows from Proposition \ref{p2.2}(iii). 
\end{proof}

\begin{corollary}\label{c3.2}
Let $U(z) = cH(z)   \Phi(z) = \sum\nolimits_{j \in E} {u_j{z^j}} $, with $u_j\in\cL_\downarrow(\ell_1)$. 
Suppose that $u_i$ is invertible in ${\cL_{\downarrow}}({\ell_1})$ for some $i$. Put  ${H_{|j-i|}} = \overline{u_j}
{z^{ i}} - \overline{u_i}{z^{j}} $ for all $j \in E\setminus \{ i\} $. Then:
\begin{romanlist}[(i)]
\item  $\left\langle {H,{H_j}(z)} \right\rangle_\Phi  = 0$ for all $j \in E\setminus \{ 0\} $;
\item 	polynomials $\{ {H_j}(z):j \in E\setminus \{ 0\} \} $ and $H(z)$ form a basis of $\cL({\ell_1})$ over ${\cL_{\downarrow}}({\ell_1})$. 
\end{romanlist}
\end{corollary}

\begin{proof} 
It follows from \eqref{eq3.1} that ${\left\langle {H,F} \right\rangle _\Phi} = 0$ if and only if $\rho (U(z)\overline {F(z)}
\,) = {\rho_U}(F(z)) = 0$ and the assertion follows from Corollary \ref{c2.1}.  
\end{proof}

Let $S_d$ be the group of permutations of variables $z_1,\ldots,z_d$ and let $\sigma \in S_d$. Then we can extend $ \sigma $ to
an automorphism of the ring $ \cL $  as follows: for any polynomial  $P(z) = \sum_{k \in {\mathbb{Z}}^d}a_k z^k$ we put 
$\sigma(P(z)) =  \sum_{k \in {\mathbb{Z}}^d} a_k  {\sigma (z^k)}$, 
where $ {\sigma (z^k)} = {\sigma (z_1^{k_1}\cdots z_d^{k_d})} = \sigma (z_1)^{k_1}\cdots \sigma (z_d)^{k_d}$. 
Thus, permutations of variables induce a  group $S_d$ of automorphisms of the ring
$ \cL $. It is also easy to check that  $ \sigma(\overline{P(z)}) =\overline{\sigma(P(z))}$.

\begin{lemma}\label{l3.2} Suppose that $ \sigma(H(z)) = H(z) $ and $ \sigma(\Phi(z)) = \Phi(z) $ 
for all $ \sigma \in S_d$. 
 Let $U(z) = cH(z) \Phi(z) = \sum\nolimits_{j \in E} {{u_j}{z^j}} $, where $u_j \in \cL_{\downarrow} $. Then:
 \begin{romanlist}[(i)]
\item  $ \sigma(U(z)) = U(z) $ for any $ \sigma \in S_d$;
\item  for any $ i \in E $ and  $ \sigma \in S_d$ there is $ j \in E $ such that  
$ \sigma(z^i) = z^j$ and $ \sigma(\overline{u_i}) = \overline{u_j}$;
\item for any $ i, j \in E $ there is  $ \sigma \in S_d$   such that  
$ \sigma(\overline{u_i}) = \overline{u_j}$ if and only if  $i$ and $j$ have the same number of zeros (unities); 
\item  $\sigma(z^i) = z^i$ and $\sigma(\overline{u_i}) = \overline{u_i}$ for any $ \sigma \in S_d$ if either     
 $i = 0 = (0,...,0)$ and $i = 1 = (1,...,1)$.
\end{romanlist}
\end{lemma}

\begin{proof}
Since $ \sigma(\overline{u_i}) = \overline{\sigma(u_i)}$, it is sufficient to prove the assertions 
(ii), (iii) and (iv) for $ u_i $ on the place of $ \overline{u_i} $. 
 
(i) Since $U(z) = cH(z) \Phi(z)$, the assertion is obvious. 

(ii) Let $ \sigma \in S_d$, it is not difficult to note that the mapping given by $ z^i \mapsto  
\sigma(z^i) $ a permutation of $ \{z^i\, |\, i \in E $ and hence there is $ j \in E $ such that  
$ \sigma(z^i) = z^j$. Since $ \sigma(U(z)) = U(z) $ we see that 
\begin{equation}\label{eq3.4}
\sum\nolimits_{i \in E} \sigma(u_i)\sigma(z^i) =  \sum\nolimits_{j \in E} {{u_j}{z^j}}.
\end{equation}
Then, as the coefficients $ u_j $ are uniquely determined, we can conclude that 
$ \sigma(u_i) = u_j$.  
  
(iii)  If there is  $\sigma \in S_d$  such that  $ \sigma(u_i) = u_j$ then \eqref{eq3.4}  
shows that  $ \sigma(z^i) = z^j$ and it easily implies that $i$ and $j$ have the same 
number of zeros (unities). 
 
Let $ i=(i_1,..., i_d) $ and $ j=(j_1,..., j_d) $ . Suppose that $ i_{k_1}=...= i_{k_m}= 
 j_{t_1}=...= j_{t_m}=1$ and $ i_{k_{m+1}}=...= i_{k_d}= 
j_{t_{m+1}}=...= j_{t_d} = 0$. Put   $$ \sigma  = 
\begin{pmatrix}
z_{i_{k_1}}&...& z_{i_{k_m}}&z_{i_{k_{m+1}}}&...&z_{i_{k_d}}\\
z_{ j_{t_1}}&...& z_{j_{t_m}}&z_{j_{t_{m+1}}}&...&z_{j_{t_d}} \\
\end{pmatrix}$$
then it is not difficult to check  that $ \sigma(z^i) = z^j$. So, \eqref{eq3.4} 
shows that $ \sigma(u^i) = u^j$
 
(iv)  The assertion follows from (iii).
 \end{proof}

\begin{proposition}\label{3.3} 
Let $U(z) = cH(z) \Phi(z) = \sum\nolimits_{j \in E} {{u_j}{z^j}} $, where $u_j \in \cL_{\downarrow} $,   
 and suppose that  $ \sigma(H(z)) = H(z) $ and $ \sigma(\Phi(z)) = \Phi(z) $ 
for all $ \sigma \in S_d$. 
Let   ${H_{|t-s|}} = \overline{u_t} {z^s} - \overline{u_s}{z^t} $, where $t \in E\setminus \{ s\}$ and 
	either $s = 0$ or $ s=1 $. Suppose that $ s $ is fixed   
 then:
 \begin{romanlist}[(i)]
\item  for any $ i \in E $ and  $ \sigma \in S_d$ there is $ j \in E $ such that  
$ \sigma(H_i) = H_j$;
\item for any $ i, j \in E $ there is  $ \sigma \in S_d$   such that  
$ \sigma(H_i) = H_j$ if and only if  $i$ and $j$ have the same number of zeros (unities). 

\end{romanlist}
\end{proposition}
\begin{proof}
	(i) Let $ i = {|t-s|} $ then $\sigma(H_i) = \sigma(\overline{u_t})\sigma({z^s}) - 
	\sigma(\overline{u_s})\sigma({z^t})$. It follows from Lemma \ref{l3.2}(iv) that 
	$ \sigma({z^s}) = {z^s}$ and $ \sigma(\overline{u_s}) = \overline{u_s}$. Then 
	$\sigma(H_i) = \sigma(\overline{u_t}){z^s} - 	\overline{u_s}\sigma({z^t})$ and it follows from 
	Lemma \ref{l3.2}(ii) that there is $ l \in E $ such that $ \sigma({z^t}) = 
	{z^l}$ and $ \sigma(\overline{u_t}) = {\overline{u_l}}$. Therefore, $\sigma(H_i) = \overline{u_l}{z^s} - 
		\overline{u_s}{z^l}$  and we can put $ j = {|l - s|}$.
		
	(ii) Suppose that  $ i = {|t-s|}$ and $ j = {|l-s|}$. At first we note that $i$ and $j$ have the same number of zeros
(unities) if and only if $t$ and $l$ have the same number of zeros (unities) . If  $ s = 0 $ then the assertion is evident. If
$ s = 1 $ then the number of zeros of $i$ and $j$  is equal to the number of unities of $t$ and $l$ and the assertion follows.
Thus, $i$ and $j$ have the same number of zeros (unities) if and only if $t$ and $l$ have the same number of zeros (unities)
and, by Lemma \ref{l3.2}(iii) if and only if there is  $ \sigma \in S_d$   such that   $ \sigma(\overline{u_t}) = \overline{u_l}$. As it was
shown in (i), $\sigma(H_i) = \sigma({\overline{u_t}}){z^s} - 	{\overline{u_s}}\sigma({z^t})$ and hence $ \sigma(H_i) = H_j =  {\overline{u_l}}{z^s} - 
{\overline{u_s}}z^l$ if and only if  $ \sigma(\overline{u_t}) = \overline{u_l}$. Thus, the assertion follows.
\end{proof}

We now provide another application of Proposition~\ref{p3.2} for $H$ of a special form related to box splines.

Evidently, the polynomial $S(z)=(1 + z^{\xi_1})   \cdots   (1 + z^{\xi_\ell})$, with $\xi_j\in\ZZ^d$,
 has zeros on the unit torus.
However, for any positive real number $\alpha $ the polynomial 
$$ {V_\alpha }(z) = \overline {(\alpha+(1 + z^{\xi_1}))   \cdots   (\alpha+(1 + z^{\xi_\ell}))} $$
has no zeros on the unit torus and hence  ${V_\alpha }(z)$ is invertible in $\cL({\ell_1})$. Evident
transformations show that the polynomial ${V_\alpha }(z)$ may be presented in the form  
\begin{equation}\label{eq3.2}
{V_\alpha }(z) = \overline {S(z)}  + \alpha {R_\alpha }(z),                              
\end{equation}
where each coefficient of the Laurent polynomial ${R_\alpha }(z)$ can be presented as a polynomial of
the variable $\alpha $ with integer coefficients that depend only on the polynomial $S$. 

\begin{lemma}\label{l3.1}
 In the above notations, suppose that $H(z) = S(z)$. Then there is a positive real number $b$ such that for any positive real
number $\alpha  < b$ the polynomial $\rho (U(z)   {V_\alpha }(z)\,)$ is invertible in ${\cL_{\downarrow}}({\ell_1})$, where
$U(z) = cH(z)\Phi(z)$. 
 \end{lemma}
 \begin{proof}
Evidently, we can assume that $0< \alpha  < 1$. It follows from \eqref{eq3.2} that $\rho (U(z)   {V_\alpha }(z)\,) = \rho (U(z) \overline {H(z)} ) + \alpha \rho (U(z)   {R_\alpha }(z))$. Then it follows from Proposition \ref{p3.1}(iii) that 
\begin{equation}\label{eq3.3}
                             \rho (U(z){V_\alpha }(z)) = 
 {2^{-d}}c\Phi({z^2}) + \alpha \rho (U(z)   {R_\alpha }(z)).                    
\end{equation}
It is also not difficult to show that each coefficient of the polynomial $\rho (U(z)   {R_\alpha }(z))$ can be presented
as a polynomial of the variable $\alpha $ with  coefficients which depend only on the polynomials $U(z)$ and $H(z)$.
This easily implies that there is a positive real number $r$ such that $\left| {\rho (U(z) {R_\alpha }(z))} \right| \le r$ on
the unit torus for any $\alpha  < 1$. On the other hand, by Proposition~\ref{p3.1}(ii), there is a positive real
number $A$ such that $\left| {\Phi(z)} \right| \ge A$. So, if we take $\alpha < b := {2^{ - d}c}{r^{ - 1}}A$,
then $\left| {2^{ - d}c}\Phi({z^2}) \right| > \left| {\alpha \rho (U(z)   {R_\alpha }(z)} \right|$ and it follows
from \eqref{eq3.3} that $\rho (U(z)   {V_\alpha }(z)\,) = {2^{ - d}}\Phi({z^2}) + \alpha \rho (U(z)   {R_\alpha }(z))$ has no
zeros on the unit torus. Therefore, the polynomial $\rho (U(z)   {V_\alpha }(z)\,)$ is invertible in
${\cL_{\downarrow}}({\ell_1})$. 
\end{proof}

\begin{corollary}\label{c3.1} 
In the above notation, suppose that $H(z) = S(z)$. Let 
$D(z) = cH(z)\Phi(z)V_{\alpha}(z) = \sum\nolimits_{j \in E}{{d_j}{z^j}} $, let ${D_j}(z) = {d_j} - {d_0}{z^{-j}}$ and let 
${H_j}(z) = \overline {{V_\alpha }(z){D_j}(z)} $, where $j \in E\setminus \{ 0\} $. 
Then there is a positive real number $b$ such that for any positive real number $\alpha  < b$ we have:
\begin{romanlist}[(i)]
\item  ${\left\langle {H,{H_j}(z)} \right\rangle _\Phi} = 0$ for all $j \in E\setminus \{ 0\} $ ;
\item  polynomials $\{ {H_j}(z):j \in E\setminus \{ 0\} \} $ and $H(z)$ form a basis of $\cL({\ell_1})$ over ${\cL_{\downarrow}}({\ell_1})$. 
\end{romanlist}
\end{corollary}
\begin{proof} 
By Lemma \ref{l3.1}, 
there is a positive real number $b$ such that for any positive real number  $\alpha  < b$ the polynomial such
that ${d_0} = \rho (U(z)  {V_\alpha }(z)\,) \ne 0$ is invertible in ${\cL_{\downarrow}}({\ell_1})$, where $D(z) =
cH(z)\Phi(z)V(z) = \sum\nolimits_{j \in E} {{d_j}{z^j}} $. Therefore,  the assertion follows from Proposition \ref{p3.2}.  
\end{proof}

\section{A module structure related to the multiresolution analysis}\label{multilevel}

Having seen the prominent
role the module $\cL(\ell_1)$ over $\cL_\downarrow(\ell_1)$ plays in the prewavelet construction, in this section we suggest 
an extension of the module structure to the entire multiresolution analysis involving the full sequence of spaces $V_j$ rather
than a  single pair of two concequtive levels $V_{-1}\subset V_0$.\par
It is easy to check that $Q_{2} =\{ x^{q} : q=\frac{m}{2^{n}},\, \text{where}\, \, m,n\in {\mathbb Z}\} $ is a torsion-free
multiplicative abelian  group. Then   $G=\times_{i=1}^{d} (Q_{2})_{i} $ is a minimax group because it has a finite series each
of whose factor is either cyclic or quasi-cyclic. 

Let $\cL_j (\ell_p)=\Big\{\sum_{k\in\ZZ^d}a_k  z^{k2^{-j}}: a\in \ell_p(\ZZ^d)\Big\}$, $1\le p\le\infty $, {and} 
$$\cL_j =\Big\{\sum_{k\in M}a_k z^{k2^{-j}} :  M\subset\ZZ^d \text{ finite, } a_k\in\CC\Big\},$$
where $j \in \mathbb{Z}$. Then $\cL_{0} =\cL$ and $\cL_{j} \le \cL_{j+1} $. It is not difficult to show that
$K=\bigcup _{j\in {\mathbb Z}}\cL_{j}  $ is a commutative domain which is isomorphic to the group ring ${\mathbb C}G$.

 Let $U_{j} =span\left\langle \varphi(2^{-j} \cdot -k)|k\in {\mathbb Z}\right\rangle $, where $j \in \mathbb{Z}$, then $U_{j}
\le U_{j+1} $. Put $U=\bigcup _{j\in {\mathbb Z}}U_{j}  $,  $V_{j} =[U_{j} ]_{l^{2} } $ and $V=\bigcup _{j\in {\mathbb{
Z}}}V_{j}  $. Then $[U]_{l^{2} } =\bigcup _{j\in {\mathbb Z}}V_{j}  =V$. 

The properties of the Fourier transformation show that 
$$\widehat{\varphi(2^{j} \cdot x-k)}(\omega ) =  \frac{1}{2^{j} } (e^{-i\omega } )^{\frac{k}{2^{j} } }
\widehat{\varphi(x)}(\frac{\varpi }{2^{j} } )=\frac{1}{2^{j} } 
z^{\frac{k}{2^{j} } }\widehat{\varphi(x)}(\frac{\varpi }{2^{j}} ),$$
where $z=e^{-i\omega } $. This easily implies that 
$\widehat{U_{j}}=\left\{F(z^{2^{-j} } )\widehat{\varphi(x)}(\frac{\varpi }{2^{j} } ):F(z^{2^{-j} } )\in \cL_{j} \right\}$ 
and hence $\widehat{U_{j} }$ is a cyclic $\cL_{j} $-module generated by $\widehat{\varphi(x)}(\frac{\varpi }{2^{j} } )$. Then,
as $\widehat{U}=\bigcup _{j\in {\mathbb Z}}\widehat{U_{j} } $, it is not difficult to show that $\widehat{U}$ is a $K$-module
generated by $\widehat{\varphi(x)}(\frac{\varpi }{2^{j}})$, where $j\in {\mathbb Z}$, and $K=\bigcup _{j\in {\ZZ}}\cL_{j}  $. 

The refinement equation shows that $\widehat{\varphi(x)}(\frac{\varpi }{2^{j-1} } )=lH(z^{2^{-j} }
)\widehat{\varphi(x)}(\frac{\varpi }{2^{j} } )$ for some integer $l$. Put $b_{j} =\prod _{s=1}^{j}lH(z^{2^{-s} } ) $. Let
$\bar{K}$ be the field of fractions of the domain $K$ and let $M$ be a $K$-submodule of $\bar{K}$ generated by the elements
$(b_{j} )^{-1} $, where $j \in {\mathbb Z}$. It is easy to see that  $(b_{j-1} )^{-1} =lH(z^{2^{-j} } )(b_{j} )^{-1} $. Since
the $K$-module $\widehat{U}$ is generated by $\widehat{\varphi(x)}(\frac{\varpi }{2^{j} } )$ with the relations
$\widehat{\varphi(x)}(\frac{\varpi }{2^{j-1} } )=lH(z^{2^{-j} } )\widehat{\varphi(x)}(\frac{\varpi }{2^{j} } )$, where $j\in
{\mathbb Z}$, it is not difficult to show that   there is a $K$-module isomorphism $\widehat{U}\cong M$. 

We note that the group ring $ \mathbb{C}G $ of an abelian torsion-free minimax group $ G $ has some algebraic properties similar to the properties of the Laurent polynomial ring $\cL$. In particular, the following analogue of Quillen-Suslin theorem holds for the group ring $ \mathbb{C}G $. \par

\begin{proposition}\label{p7.1}
Let $G$ be an abelian torsion-free minimax group and $ U $  be a projective $\mathbb{C} G$-module which is a direct summand of
a finitely generated free $\mathbb{C} G$-module. Then $ P $ is a finitely generated free $\mathbb{C} G$-module. 
\end{proposition}

\begin{proof}
Let $ W$ be a finitely generated free $\mathbb{C}G$-module such that $ W = V \oplus  U $.  Let $A= \{a_1,...,\, a_n\}$  be a
finite set of free generators of the  $\mathbb{C}G$-module $W$. Then $a_i= b_i +c_i$, where $b_i \in V$ and $c_i \in U$ for
each $i \in \{1,...,\, n\}$.  Put $B= \{b_1,...,\, b_n\}$ and $C= \{c_1,...,\, c_n\}$. Since the group $ G $ is countable, it
is a union of an ascending chain $\{G_i \mid i \in \mathbb{N}\}$ of its finitely generated subgroups $G_i$. Therefore, $W=
\cup_{i \in \mathbb{N}}A\mathbb{C}G_i $ and hence there is a finitely generated subgroup $H$  of  $G$ such that $B,
C\subseteq   A\mathbb{C}H$. It easily implies that $ A\mathbb{C}H =  B\mathbb{C}H \oplus   C\mathbb{C}H$. It is well known
that $\mathbb{C}H\cong \cL$ and hence, by Quillen-Suslin theorem,   $B\mathbb{C}H$ and $C\mathbb{C}H$ are free  $\mathbb{C}
H$-modules. Let $\tilde{B}$ and $\tilde{C}$ be sets of free generators for  $B\mathbb{C} H$ and $C\mathbb{C} H$ respectively
then $|\tilde{B}|+|\tilde{C}| =n $.  It is not difficult to show that $ V= \tilde{B}\mathbb{C}G$ and $ U=
\tilde{C}\mathbb{C}G$. Since the group $G$ is torsion-free abelian, $\mathbb{C}G$ is a domain and hence it has a field of
fraction $\mathbb{F}$. Then $dim_{ \mathbb{F}}({V\otimes_{\mathbb{C}G}\mathbb{F}})\leq |\tilde{B}|$, $dim_{
\mathbb{F}}({U\otimes_{\mathbb{C}G}\mathbb{F}})\leq |\tilde{C}|$ and $dim_{ \mathbb{F}}({W\otimes_{\mathbb{C}G}\mathbb{F}})=
|A| = n$. Since $W\otimes_{\mathbb{C}G}\mathbb{F}  = ({V\otimes_{\mathbb{C}G}\mathbb{F}})\oplus
({U\otimes_{\mathbb{C}G}\mathbb{F}})$ and $|\tilde{B}|+|\tilde{C}| =n $, it follows that $dim_{
\mathbb{F}}({V\otimes_{\mathbb{C}G}\mathbb{F}})= |\tilde{B}|$ and $dim_{ \mathbb{F}}({U\otimes_{\mathbb{C}G}\mathbb{F}}) =
|\tilde{C}|$. Therefore, elements of $|\tilde{C}|$ are linearly independent over $\mathbb{C}G$ and, as $ U=
\tilde{C}\mathbb{C}G$, we can conclude that $ \tilde{C} $ is a set of free generators of $U$.  
\end{proof}

\section{Construction of box spline prewavelets of small support}
\label{Section 7}

Consider the multiresolution analysis generated by a box spline $\varphi=M_\Xi$ defined in
Section~\ref{prewBox}, assuming that $\Xi$ is unimodular.
Then $H$ and $\Phi$ are polynomials given by \eqref{Hbox} and \eqref{Phibox}, respectively.
In this section we construct Laurent polynomials $H_s$, $s \in E \setminus \{ 0\}$, such that
$ \{ H \} \cup  \{ H_s : s \in E \setminus \{ 0\} \}$ is a basis of the module $\cL(\ell_1)$ 
over $\cL_{\downarrow}(\ell_1)$, and $ {\left\langle {H,H_s} \right\rangle_\Phi} = 0$ for all $s\in E \setminus \{ 0\}\}$ 
in some important special cases of $\Xi$. 

The method of Section~\ref{Section 3} provides two different waysto construct box spline prewavelets, 
see Corollaries~\ref{c3.2} and \ref{c3.1}. The construction of Corollary~\ref{c3.1} is applicable to all box splines with a 
unimodular $\Xi$. However, it leads to prewavelets with relatively large supports. 
In contrast, Corollaries~\ref{c3.2} is only available if at least one of the polynomials $u_i$ is invertible in 
$\cL_{\downarrow}(\ell_1)$. Yet, in all special cases considered below this condition is satisfied and 
leads to prewavelets with remarkably small mask supports.

As in Section~\ref{Section 3}, we set $U := c H\Phi$ for an appropriate nozero coefficient $c\in \mathbb{C}$. 
Below we will always choose $c$ as the smallest positive integer such that all coefficients of $ U $
are integers.

Consider the expansion $U(z)  = \sum_{s \in E} {{u_s}{z^s}} $. 
It is proved below that in the considered cases there exists  $ {s_0 \in E} $ such that the polynomial   
$u_{s_0} $ has no zeros on the unit torus, i.e.\ $ u_{s_0} $ is invertible in $ \cL_{\downarrow}(\ell_1) $. Therefore,  
by Corollary \ref{c3.2}, we can choose 
${H_{|r-s_0|}} = \overline{u_r} {z^{ s_0}} - \overline{u_{s_0}}{z^{r}} \quad\text{for all}\quad r \in E\setminus \{ s_0\}$,
which defines all $H_s$ with $s\in E\setminus \{ 0\}$. To simplify the notation we set $W(z) = \overline{U(z)}$  
and $w_j = \overline{u_j}$. Then
\begin{equation}\label{eq4.1}
{H_{|r-s_0|}} = {w_r} {z^{ s_0}} - {w_{s_0}}{z^{r}} \quad\text{for all}\quad r \in E\setminus \{ s_0\}, 
\end{equation}

Since 
$$W(z)=\overline{U(z)} = %
\overline{cH(z)} \Phi(z) =\sum_{j \in E} {\overline{u_j}{z^{-j}}}, $$ 
we can easily obtain   the 
polynomials $w_j = \overline{u_j} $ after calculating $\Phi(z)$ and expanding the product  in
$$\overline{U(z)} =  %
\overline{c}2^{-\ell}(1 + {z^{-\xi_1}})  \cdots  (1 + z^{-\xi_\ell})\, \Phi(z).$$  

For any $F \in \cL$, the set of monomials of $F$ with nonzero coefficients is said to be the support $ \supp F $ of the
polynomial $ F $. 

In our presentation of basis polynomials we also use the results of Proposition 
\ref{3.3} about relations between  the polynomials determined by 
 (\ref{eq4.1}) and permutations of variables. The assumption that $ \sigma(H(z)) = H(z) $ and $ \sigma(\Phi(z)) = \Phi(z) $ 
for all $ \sigma \in S_d$ is satisfied in all cases considered below, which is easily seen by Lemma~\ref{permbox}. 
For example, in the two-dimensional case $H_{0,1}(x,y)=H_{1,0}(y,x)$ and  $H_{1,1}(x,y)=H_{1,1}(y,x)$. 

For the calculations with multivariate polynomials we used a computer algebra system SageMath  (\url{https://www.sagemath.org/})

\subsection{Piecewise linear prewavelets on the three-direction mesh}\label{linear}
Let
$$
\Xi=\begin{bmatrix}
1 & 0 & 1\\
0 & 1 & 1
\end{bmatrix},$$
in which case $\varphi=M_\Xi$ is a bivariate Courant hat function and $V_0$ 
consists of square integrable continuous piecewise linear functions on the three direction mesh \cite{BHR93}.
Using the notation $x=z_1$, $y=z_2$, we have 
\begin{align*}
H &= \tfrac18(1+x)(1+y)(1+ xy),\\ 
\Phi &= \tfrac1{12}(6+ x+ x^{-1}+ y+ y^{-1}+ x^{-1} y^{-1}+ x y), \quad \text{see \cite[Section 6]{JM}}.
\end{align*}
We choose the coefficient $ c = 96$. %
A simple calculation shows that 
$$W = 
\overline{cH}\Phi = w_{0,0} +w_{1,0}x^{-1}+w _{0,1}y^{-1}+w _{1,1}x^{-1}y^{-1},
\qquad w_{s,t}\in \cL_{\downarrow},$$ 
where %
\begin{align*} 
w_{0,0} &= 10 +2y^{-2} + 2x^{-2} + 10x^{-2}y^{-2},\\ 
w_{1,0} &= 2x^2 + 10 + 10y^{-2} + 2x^{-2}y^{-2},\\
w_{0,1} &= 2y^2 + 10 + 10x^{-2} + 2x^{-2}y^{-2},\\
w_{1,1} &= x^2y^2 + x^2 + y^2 + 18 + y^{-2} + x^{-2} + x^{-2}y^{-2}, 
\end{align*}
with $w_{0,1}(x,y)=w_{1,0}(y,x)$, $w_{0,0}(y,x)=w_{0,0}(x,y)$ and $w_{1,1}(y,x)=w_{1,1}(x,y)$ in agreement with 
Lemma~\ref{l3.2}.

\begin{proposition}\label{7.1} 
$u_{1,1}$ has no zeros on the unit torus. 
\end{proposition}
\begin{proof}
We observe that $\overline{w_{1,1}}= w_{1,1}= u_{1,1} $ on the unit torus.
Hence $u_{1,1}$ takes only  real values there. As 
$|x| = |y| = 1$ on the unit torus, it is not difficult to notice that $|w_{1,1} - 18|\leq  6$, 
 which implies that $u_{1,1}\ge12$. 
Thus, we conclude that $u_{1,1}$ has no zeros on the unit torus.  
\end{proof}

Thus, by Corollary  \ref{c3.2}, we apply (\ref{eq4.1}) and obtain generators of prewavelets in the form
\begin{align*} 
H_{1,0} & =w_{0,1}xy - w_{1,1}y \\
& = -x^2y^3 + 2xy^3 - x^2y - y^3 + 10xy - 18y + 10x^{-1}y - y^{-1}- x^{-2}y + 2x^{-1}y^{-1} - x^{-2}y^{-1},\\
H_{0,1}&=w_{1,0}xy - w_{1,1}x \\
& = -x^3y^2 + 2x^3y - x^3 - xy^2 + 10xy - 18x + 10xy^{-1} - xy^{-2} - x^{-1}+ 2x^{-1}y^{-1} - x^{-1}y^{-2},\\
H_{1,1}&=w_{0,0}xy - w_{1,1} \\
& = -x^2y^2 - x^2 + 10xy - y^2 + 2xy^{-1} - 18 + 2x^{-1}y - y^{-2} + 10x^{-1}y^{-1} -  x^{-2} - x^{-2}y^{-2}.
\end{align*}
The coefficients of these polynomials are illustrated in the following matrices whose entries contain the coefficients
corresponding to the monomials $x^iy^j$ with increasing $i$ from left to right and increasing $j$ from bottom to top, as in a
Cartesian coordinate system, and the coefficient of the constant monomial $x^0y^0$ is framed.
$$ \cH_{1,0}= 
\begin{bmatrix}
0&0&-1&2&-1\\
0&0&0&0&0\\
-1&10&-18&10&-1\\
0&0&\framebox{0}&0&0\\
-1&2&-1&0&0\\
\end{bmatrix},
\qquad 
\cH_{0,1}= 
\begin{bmatrix}
0&0&-1&0&-1\\
0&0&10&0&2\\
-1&\framebox{0}&-18&0&-1\\
2&0&10&0&0\\
-1&0&-1&0&0\\
\end{bmatrix},
$$
$$ \cH_{1,1}= 
\begin{bmatrix}
0&0&-1&0&-1\\
0&2&0&10&0\\
-1&0&\framebox{-18}&0&-1\\
0&10&0&2&0\\
-1&0&-1&0&0\\
\end{bmatrix}.
$$
In particular, in accordance with Proposition \ref{3.3}, $H_{0,1}(x,y)=H_{1,0}(y,x)$ and $H_{1,1}(x,y)=H_{1,1}(y,x)$. 
Note that the permutation of the variables leads to the transposition of the matrix 
relative to the anti-diagonal going through the framed entry.

We observe that $|\supp H_{(i,j)}| = 11$ for all ${(i,j)} \in E\setminus {(0,0)}$. 
Let $N$ denote the total number of nonzero coefficients of the polynomials $H_s$,
$$
N=\sum_{s\in E\setminus\{0\}}|\supp H_s|.$$
Thus, $N=33$ for our construction.
As it is adavantageous for the applications to have $N$ as small as possible, we compare this result with the
constructions known from the literature.

Polynomials $H_s$ of  prewavelets constructed in
\cite{KO95} have either 13 or 10  nonzero coefficients each, leading to $N=39$ or 30. 
Alternative constructions have significantly larger
$N$, with $N=69$ in \cite[Section 6]{JM} and \cite[Section 4]{Lai06}.

\subsection{Bivariate $C^1$ piecewise cubic box spline prewavelets}\label{semi-square}
Let
$$
\Xi=\begin{bmatrix}
1 & 1 & 0 & 0 & 1\\
0 & 0 & 1 & 1 & 1
\end{bmatrix}.$$
Then the box spline $M_\Xi$ and its translates are $C^1$ piecewise cubic functions on the three-direction mesh,
$$H =2^{-5} (1+x)^2 (1+y)^2  (1+xy),$$ 
and,
by \cite[p. 35]{RS}, 
\begin{align*}
\Phi =\frac{1}{10080} &\big(3194 + 31x^2 + 31x^{-2} + 31y^2 + 31y^{-2} + 5x^2y^2  + 5x^{-2}y^{-2}+1144x\\
&+1144x^{-1}+47xy^2+1xy^{-2}+1x^{-1}y^2+47x^{-1}y^{-2}+1144y+1144y^{-1}\\
&+47yx^2+1yx^{-2}+1y^{-1}x^2+47y^{-1}x^{-2}+814xy+814x^{-1}y^{-1}+178xy^{-1}+178x^{-1}y\big).
\end{align*}
With $c=2^5\cdot 10080$ we obtain
$W =\overline{cH} \Phi= w_{0,0} +w_{1,0}x^{-1}+w _{0,1}y^{-1}+w _{1,1}x^{-1}y^{-1}$ ,
where  
\begin{align*}
w_{0,0} &= 5x^2y^2 + 130x^2 + 130y^2 + 33x^2y^{-2} + 12303 + 33x^{-2}y^2 + 12048y^{-2} + 12048x^{-2}\\
&\quad + 215y^{-4} + 35193x^{-2}y^{-2} + 215x^{-4} + 3630x^{-2}y^{-4} + 3630x^{-4}y^{-2} + 1027x^{-4}y^{-4},\\
w_{1,0} &= 57x^2y^2 + 3136x^2 + 110y^2 + 1677x^2y^{-2} + 18425 + x^{-2}y^2 + 2x^2y^{-4} + 30692y^{-2}\\
&\quad + 3134x^{-2} +1677y^{-4} + 18425x^{-2}y^{-2} + x^{-4} + 3136x^{-2}y^{-4} + 110x^{-4}y^{-2} + 57x^{-4}y^{-4},\\
w_{1,1} &= 1027x^2y^2 + 3630x^2 + 3630y^2 + 215x^2y^{-2} + 35193 + 215x^{-2}y^2 + 12048y^{-2} + 12048x^{-2}\\
&\quad + 33y^{-4} +
12303x^{-2}y^{-2} + 33x^{-4} + 130x^{-2}y^{-4} + 130x^{-4}y^{-2} + 5x^{-4}y^{-4},
\end{align*}
and $w_{0,1}(x,y)=w_{1,0}(y,x)$.

\begin{proposition}\label{7.2}
 The polynomial $u_{0,0} =\overline{w_{0,0}}$ has no zeros on unit torus.  
 \end{proposition}
 
 \begin{proof}
	  For $(x,y)\in\TT^2$ we have $|x| = |y| = 1$ .
Since $u_{0,0} = e_{0, 0}x^2y^2$, where \\ 
$e_{0, 0} = 1027x^2y^2 + 3630x^2 + 3630y^2 + 215x^2y^{-2} + 35193 + 215x^{-2}y^2 + 12048y^{-2} + 12048x^{-2} + 33y^{-4} + 12303x^{-2}y^{-2} + 33x^{-4} + 130x^{-2}y^{-4} + 130x^{-4}y^{-2} + 5x^{-4}y^{-4}$, \\
it is sufficient to show that that $e_{0, 0}$ has no zeros on unit torus.   
Put $x^2 = u$ and $y^2 = v$ then $|u| = |v| = 1$. Then \\ 
$e_{0, 0} = 1027uv+ 3630u + 3630v + 215uv^{-1} + 35193 + 215u^{-1}v+ 12048v^{-1} + 12048u^{-1} + 33v^{-2} + 12303u^{-1}v^{-1} + 33u^{-2} + 130u^{-1}v^{-2} + 130u^{-2}v^{-1} + 5u^{-2}v^{-2}$. \\
Since $12303u^{-1}v^{-1} = 12048u^{-1}v^{-1} +255u^{-1}v^{-1} $, we can presented $e_{0, 0}$ in the following form: \\
$$e_{0, 0} = 35193 +12048a(u,v)+b(u,v),$$ 
where $a(u,v)= u^{-1}+v^{-1}+u^{-1}v^{-1}$ and\\ 
$b(u,v) = 1027uv+ 3630u + 3630v + 215uv^{-1} + 215u^{-1}v+  33v^{-2} + 255u^{-1}v^{-1} + 33u^{-2} + 130u^{-1}v^{-2} +
130u^{-2}v^{-1} + 5u^{-2}v^{-2}$.\\ 
As $|u| = |v| = 1$, we see that $|b(u,v)| \leq  1027|uv|+ 3630|u| + 3630|v |+ 215|uv^{-1}| + 215|u^{-1}v|+  33|v^{-2}| +
255|u^{-1}v^{-1}| + 33|u^{-2}| + 130|u^{-1}v^{-2}| + 130|u^{-2}v^{-1}| + 5|u^{-2}v^{-2}| = 1027+ 3630 + 3630+ 215+ 215+ 33 +
255 + 33 + 130 + 130+ 5 \leq 9303$. Thus, $|b(u,v)| \leq  9303$ and hence $\Ree b(u,v) \in [-9303, 9303]$. 

Since $|u| = |v| = 1$, the complex numbers $u^{-1}$ and $v^{-1}$ have the following trigonometric form: $u^{-1}= \cos \alpha +
i\sin \alpha $ and $v^{-1}=\cos \beta+ i\sin \beta $. Therefore $\Ree a(u,v)=\cos \alpha +\cos \beta + \cos \alpha \cos \beta - \sin \alpha
\sin \beta = (1+\cos \alpha)  (1+\cos \beta) - (1+\sin \alpha \sin \beta) $. Evidently, $(1+\cos \alpha)  (1+\cos \beta) \in [0, 4]$ and
$(1+\sin \alpha \sin \beta) \in [0,2]$ and hence $ \Ree a(u,v)\in [-2, 4]$. As $e_{0, 0} = 35193 +12048a(u,v)+b(u,v)$ and
$\Ree b(u,v) \in [-9303, 9303]$, it implies that $ \Ree  e_{0, 0} \geq 1794$ for all complex numbers $x$ and $y$ such that $|x| =
|y| = 1$. Thus, $e_{0,0}$ has no zeros on unit torus.   
\end{proof}

Then, by Corollary \ref{c3.2}, we apply (\ref{eq4.1}) and obtain in the matrix notation,
$$ \cH_{1,0}= 
\begin{bmatrix}
0&0&1&-33&110&-130&57&-5\\
0&0&0&0&0&0&0&0\\
1&-215&3134&-12048&\framebox{18425}&-12303&3136&-130\\
0&0&0&0&0&0&0&0\\
110&-3630&18425&-35193&30692&-12048&1677&- 33\\
0&0&0&0&0&0&0&0\\
57&-1027& 3136&-3630&1677&-215&2&0\\
\end{bmatrix},$$
$$ \cH_{1,1}= 
\begin{bmatrix}
0&0&0&-33&0&-130&0&-5\\
0&0&215&0&3630&0&1027&0\\
0&-215&0&-12048&0&-12303&0&-130\\
33&0&12048&0&\framebox{35193}& 0&3630&0\\
0&-3630&0&-35193&0&-12048&0&-33\\
130&0&12303&0&12048&0&215&0\\
0&-1027&0&-3630&0&-215&0&0\\
 5&0&130&0&33&0&0&0\\
\end{bmatrix}
$$
and $H_{0,1}(x,y)=H_{1,0}(y,x)$.

We see that $|\supp H_{1,0} | = |\supp H_{0,1} | =29$ and $|\supp H_{1,1} | =  28$, hence for our construction
$N=86$.
Polynomials $H_{i,j}$ for the prewavelets constructed in \cite{RS} have 58 nonzero coefficients, so that $N=174$ there.

\subsection{Bivariate $C^2$ piecewise quartic box spline prewavelets}\label{square}
Let
$$
\Xi=\begin{bmatrix}
1 & 1 & 0 & 0 & 1 & 1\\
0 & 0 & 1 & 1 & 1 & 1
\end{bmatrix}.$$
Then the box spline $M_\Xi$ and its translates are $C^2$ piecewise quartic functions on the three-direction mesh,
$$H = 2^{-6} (1+x)^2 (1+y)^2  (1+xy)^2,$$
and, by \cite[p. 36]{RS}, 
\begin{align*}
\Phi = \frac{1}{322560} &\big(2x^{-3} + 1736x^{-2} +37742x^{{-1}} + 94992  + 37742 x+1736 x^2+2 x^3+ 34x^{-2}y \\
&+ 5100x^{{-1}}y + 37742y + 37742xy + 5100x^2y + 34x^3y + 34x^{{-1}}y^2 + 1736y^2\\
& + 5100xy^2 + 1736x^2y^2 + 34x^3y^2 +2y^3  +  34xy^3  +  34x^2y^3 + 2x^3y^3 + 34x^{-3}y^{{-1}}\\
& + 5100x^{-2}y^{{-1}} +37742x^{{-1}}y^{{-1}} + 37742y^{{-1}} + 5100x^1y^{{-1}}  + 34x^2y^{{-1}} + 34x^{-3}y^{-2}\\
& + 1736x^{-2}y^{-2} + 5100x^{{-1}}y^{-2} + 1736y^{-2} + 34xy^{-2} +  2 x^{-3}y^{-3}  +  34x^{-2}y^{-3} \\
& + 34x^{-1}y^{-3} + 2y^{-3}\big). 
\end{align*}
With $ c=2^6\cdot 322560$ we get
$W =\overline{c H}\Phi= w_{0,0} +w_{1,0}x^{{-1}}+w _{0,1}y^{{-1}}+w _{1,1}x^{{-1}}y^{{-1}}$, 
where    
\begin{align*}
w_{0,0} & =  1884x^2y^2 + 14088x^2 + 14088y^2 + 1884x^2y^{-2} + 555000 + 1884x^{-2}y^2 + 555000y^{-2}\\
&\quad + 555000x^{-2} +14088y^{-4} + 2380248x^{-2}y^{-2} + 14088x^{-4} + 555000x^{-2}y^{-4}\\
&\quad + 555000x^{-4}y^{-2} + 1884x^{-2}y^{-6} + 555000x^{-4}y^{-4} + 1884x^{-6}y^{-2} + 14088x^{-4}y^{-6}\\
&\quad  + 14088x^{-6}y^{-4} + 1884x^{-6}y^{-6},
\end{align*}
\begin{align*}
w_{1,0} & =  38x^4y^2 + 104x^4 + 8890x^2y^2 + 2x^4y^{-2} + 163440x^2 + 8890y^2 + 69054x^2y^{-2}\\
&\quad + 811088 + 38x^{-2}y^2 +104x^2y^{-4} + 1677952y^{-2} + 163440x^{-2} + 163440y^{-4}\\
&\quad + 1677952x^{-2}y^{-2} + 104x^{-4} + 38y^{-6} + 811088x^{-2}y^{-4} +69054x^{-4}y^{-2} + 8890x^{-2}y^{-6}\\
&\quad + 163440x^{-4}y^{-4} + 2x^{-6}y^{-2} + 8890x^{-4}y^{-6} + 104x^{-6}y^{-4} +38x^{-6}y^{-6},\\
w_{1,1} & =  2x^4y^4 + 104x^4y^2 + 104x^2y^4 + 38x^4 + 69054x^2y^2 + 38y^4 + 163440x^2 + 163440y^2\\
&\quad  + 8890x^2y^{-2} + 1677952 +8890x^{-2}y^2 + 811088y^{-2} + 811088x^{-2} + 8890y^{-4}\\
&\quad  + 1677952x^{-2}y^{-2} + 8890x^{-4} + 163440x^{-2}y^{-4} +163440x^{-4}y^{-2} + 38x^{-2}y^{-6}\\
&\quad  + 69054x^{-4}y^{-4} + 38x^{-6}y^{-2} + 104x^{-4}y^{-6} + 104x^{-6}y^{-4} +2x^{-6}y^{-6}
\end{align*}
and $w_{0,1}(x,y)=w_{1,0}(y,x)$.

\begin{proposition}\label{7.3}
The polynomial $u_{0,0} = \bar{w}_{0,0}$ has no zeros on the  unit torus and hence $u_{0,0}$  is invertible  in $\cL_{\downarrow}(\ell_1)$.  
\end{proposition}
\begin{proof}	
Let $e_{0, 0} = u_{0, 0} x^{-2}y^{-2}$; then $u_{0, 0}$  has zeros on the unit torus if and only if $e_{0, 0}$ has such
zeros. 
Put $x^2 =u$ and $y^2 = v$, then 
$e_{0, 0} = 2380248 +555000a(u, v) + 1884b(u, v) +14088c(u, v)$, 
where 
\begin{align*}
a(u, v) &= uv + u + v+ u^{-1} + v^{-1} + u^{-1}v^{-1},\\
b(u, v) &= u^2v^2 +u^2+ v^2 +u^{-2} + + v^{-2} +u^{-2}v^{-2}, \\
c(u, v) &= u^2v + uv^2 + uv^{-1} + u^{-1}v^{-2} +  u^{-1}v + u^{-2}v^{-1},
\end{align*}
and  $|u| = |v| = 1$ on the unit torus.
 
One can see that $a(u, v) = \overline{a(u, v)}$, $b(u, v)  = \overline{b(u, v)}$, $c(u, v)  = \overline{c(u, v)}$  and hence
the values of $a(u, v)$, $b(u, v)$, and $c(u, v)$ are real numbers. Since $|u| = |v| = 1$, we see that $a(u, v), b(u, v), c(u,
v) \in [-6, 6]$ on the unit torus.  

It is easy to see that $b(u, v) = a(u^2, v^2)$ and $c(u, v) = a(uv^2, uv^{-1})$. 
It is not difficult to show that $a(u, v) = (u + 1)(v +1) + (u^{-1} + 1)(v^{-1} +1) - 2$. 
Since $|u| = |v| = 1$, we see that $\Ree (u + 1), \Ree (v + 1) \in [0, 2]$ and 
$\Imm(u + 1), \Imm (v + 1) \in [-1, 1]$. Therefore, as 
$\Ree ((u + 1)(v +1)) = \Ree (u + 1) \Ree (v + 1) - \Imm(u + 1) \Imm(v + 1)$, 
we can conclude that $\Ree ((u + 1)(v +1))  \in [ {-1}, 4]$. Since 
$a(u, v)  = (u + 1)(v +1) + (u^{-1} + 1)(v^{-1} +1) - 2$, it easily implies that 
$a(u, v)  \in [ {-4}, 6]$. As $b(u, v) = a(u^2, v^2)$ and $c(u, v) = a(uv^2, uv^{-1})$, the above arguments show that 
$b(u, v), c(u, v) \in [ {-4}, 6]$. Thus, $a(u, v), b(u, v) , c(u, v)  \in [ {-4}, 6]$ and, as 
$e_{0, 0} = 2380248 +555000a(u, v) + 1884b(u, v) +14088c(u, v)$, we can see that $e_{0, 0}(x, y)$ is a real function and
$e_{0, 0}(x, y) \in [96360, 5806080]$ on the unit torus. Thus, we can conclude that $u_{0, 0}$ has no zeros on the unit torus.
\end{proof}

		Thus, by Corollary \ref{c3.2}, we can apply (\ref{eq4.1}) and hence obtain $H_{i,j}$ in the matrix notation
$$\cH_{1,0}={\scriptsize\begin{bmatrix}
0&0&0&0&38&-1884&8890&-14088&8890&-1884&38\\
0& 0&0&0&0&0&0&0&0&0&0\\
0&0&104&-14088&163440&-555000&\framebox{811088}&-555000&163440&-14088&104\\
0&0&0&0&0&0&0&0&0&0&0\\
2&-1884&69054&-555000&1677952&-2380248&1677952&-555000&69054&-1884&2\\
0&0&0&0&0&0&0&0&0&0&0\\
140&-14088&163440&-555000&811088&-555000&163440&-14088&104&0&0\\
0&0&0&0&0&0&0&0&0&0&0\\
38&-1884&8890&-14088&8890&-1884&38&0&0&0&0%
\end{bmatrix}},$$
$$\cH_{1,1}= {\scriptsize\begin{bmatrix}
0&0&0&0&0&0&38&0&104&0&2\\
0&0&0&0&0&-1884&0&-14088&0&-1884&0\\
0&0&0&0&8890&0&163440&0&69054&0&104\\
0&0&0&14088&0&-555000&0&-555000&0&-14088&0\\
0&0&8890&0&811088&0&\framebox{1677952}&0&163440&0&38\\
0&-1884&0&-555000&0&-2380248&0&-555000&0&-1884&0\\
38&0&163440&0&1677952&0&811088&0&8890&0&0\\
0&-14088&0&-555000&0&-555000&0&-14088&0&0&0\\
104&0&69054&0&163440&0&8890&0&0&0&0\\
0&-1884&0&-14088&0&-1884&0&0&0&0&0\\
2&0&104&0&38&0&0&0&0&0&0%
\end{bmatrix}},$$
and $H_{0,1}(x,y)=H_{1,0}(y,x)$.

Thus, $|\supp H_{i,j} | =43$ for all  $i,j$, hence $N=129$ for our construction.
The prewavelets constructed in \cite{RS} and \cite{BDG} have 91, respectively 51, nonzero coefficients, with 
$N=273$ or 153.

\subsection{Trivariate piecewise linear prewavelets}\label{3-dim}
Let
$$
\Xi=\begin{bmatrix}
1 & 0 & 0 & 1\\
0 & 1 & 0 & 1\\
0 & 0 & 1 & 1
\end{bmatrix}.$$
Then $M_\Xi$ and its integer translates generate a space $V_0$ of trivariate continuous piecewise linear functions, see 
\cite{BHR93}, and
$$
H = 2^{-4}(1+x)(1+y)(1+z) (1+xyz),$$ 
where we use the notation $x=z_1$, $y=z_2$, $z=z_3$.
To compute the function $\Phi$ given by \eqref{Phibox}, we evaluated the centered box spline $M_{\Xi\cup \Xi}^c(k)$ at points  
$k\in\ZZ^3$ in its support using the recurrence relation \cite[p.~17]{BHR93} and the code provided in
\cite{deBoor93}, which gives 
\begin{align*} 
\Phi =  \frac1{60}x^{-2}y^{-2}z^{-2}&\big(3xyz + 3x^3y^3z^3 + 3xy^2z^2 + 3x^2yz^2 + 3x^2y^2z + 3x^3y^2z^2 + 3x^2y^3z^2 +
3x^2y^2z^3\\
& + 2x^2yz + 2xy^2z + 2xyz^2 + 2x^2y^3z^3 + 2x^3y^2z^3 + 2x^3y^3z^2 + 24x^2y^2z^2\big),
\end{align*}
and it follows that 
\begin{align*}
W =\overline{ c H}\Phi &= 
w_{0,0,0} + w_{1,0,0}x^{{-1}} + w_{0,1,0}y^{{-1}} + w_{0,0,1}z^{{-1}} + w_{1,1,0}x^{{-1}}y^{{-1}}\\
&\quad + w_{1,0,1}x^{{-1}}z^{{-1}} + w_{0,1,1}y^{{-1}}z^{{-1}} + w_{1,1,1}x^{{-1}}y^{{-1}}z^{{-1}},
\end{align*}
where  $ c = 2^4\cdot 60 $ and
\begin{align*}
w_{0,0,0} &=  45 + 45x^{-2}y^{-2}z^{-2}+5x^{-2}z^{-2} + 5y^{-2} + 5y^{-2}z^{-2} +5x^{-2}y^{-2} + 5x^{-2}  + 5z^{-2}, \\
w_{1,0,0} &= 10x^2 + 40 + 10z^{-2} + 10y^{-2} + 40y^{-2}z^{-2} + 10x^{-2}y^{-2}z^{-2},\\
w_{1,1,0} &= 5x^2y^2 + 5x^2 + 5y^2 + 45 + 45z^{-2} + 5y^{-2}z^{-2} + 5x^{-2}z^{-2} + 5x^{-2}y^{-2}z^{-2},\\
w_{1,1,1} &= 3x^2y^2z^2 + 2x^2y^2 + 2x^2z^2 + 2y^2z^2 + 3x^2 + 3y^2 + 3z^2 + 84 + 3z^{-2} + 3y^{-2} + 3x^{-2}\\
&\quad + 2y^{-2}z^{-2} + 2x^{-2}z^{-2} + 2x^{-2}y^{-2} + 3x^{-2}y^{-2}z^{-2},
\end{align*}
with $w_{0,1,0}(x,y,z)=w_{1,0,0}(y,x,z)$, $w_{0,0,1}(x,y,z)=w_{1,0,0}(z,y,x)$, $w_{1,0,1}(x,y,z)=w_{1,1,0}(x,z,y)$, 
$w_{0,1,1}(x,y,z)=w_{1,1,0}(z,y,x)$, in accordance with Lemma~\ref{l3.2}.

\begin{proposition}\label{p7.4} 
The polynomial $u_{1,1,1} = \bar{w}_{1,1,1} $ has no roots on the unit torus.  
\end{proposition}
	\begin{proof}
By estimating the absolute value of each non-constant term in $w_{1,1,1}$ using the fact that 
$|x| = |y| = |z| = 1$  on the unit torus, we conclude that $|w_{1,1,1}-84|\le 36$, and hence $u_{1,1,1}\ne 0$.
\end{proof}
Thus, by Corollary \ref{c3.2}, we can apply (\ref{eq4.1}). Therefore, 
\begin{align*} 
H_{1,0,0} &= w_{0,1,1}xyz - w_{1,1,1}yz\\
&= -3x^2y^3z^3 + 5xy^3z^3 - 2x^2y^3z - 2x^2yz^3 - 2y^3z^3 + 5xy^3z + 5xyz^3 - 3x^2yz - 3y^3z\\
&\quad - 3yz^3 + 45xyz - 84yz + 45x^{-1}yz- 3yz^{-1} - 3y^{-1}z - 3x^{-2}yz + 5x^{-1}yz^{-1} + 5x^{-1}y^{-1}z\\
&\quad - 2y^{-1}z^{-1} - 2x^{-2}yz^{-1} - 2x^{-2}y^{-1}z + 5x^{-1}y^{-1}z^{-1} - 3x^{-2}y^{-1}z^{-1}.
\end{align*}
To illustrate the structure of this polynomial we use matrices of coefficients
corresponding to the monomials $x^iy^j$ as in previous subsections,
with the coefficient of the constant monomial $x^0y^0$ framed,
$$\cH_{1,0,0} =
{\scriptsize\begin{bmatrix}
-2&5&-3\\
0&0&\framebox{0}\\
-3&5&-2\\
\end{bmatrix}}z^{-1}+
{\scriptsize\begin{bmatrix}
0&0&-3&5&-2\\
0&0&0&0&0\\
-3&45&-84&45&-3\\
0&0&\framebox{0}&0&0\\
-2&5&-3&0&0\\
\end{bmatrix}}z+
{\scriptsize\begin{bmatrix}
-2&5&-3\\
0&0&0\\
-3&5&-2\\
\framebox{0}&0&0\\
\end{bmatrix}}z^3.$$
According to Proposition~\ref{3.3}, $H_{0,1,0}(x,y,z)=H_{1,0,0}(y,x,z)$ and $H_{0,0,1}(x,y,z)=H_{1,0,0}(z,y,x)$, and we see
that $|\supp H_{1,0,0}|=|\supp H_{0,1,0}|=|\supp H_{0,0,1}|=23$.

Furthermore,
\begin{align*} 
H_{1,1,0} &= w_{0,0,1}xyz - w_{1,1,1}z \\
&= -3x^2y^2z^3 - 2x^2y^2z - 2x^2z^3 + 10xyz^3 - 2y^2z^3 - 3x^2z + 40xyz - 3y^2z - 3z^3\\
&\quad + 10xy^{-1}z - 84z + 10x^{-1}yz - 3z^{-1}- 3y^{-2}z + 40x^{-1}y^{-1}z - 3x^{-2}z\\
&\quad - 2y^{-2}z^{-1} + 10x^{-1}y^{-1}z^{-1} - 2x^{-2}z^{-1} - 2x^{-2}y^{-2}z -3x^{-2}y^{-2}z^{-1},
\end{align*}
such that
$$\cH_{1,1,0} =
{\scriptsize\begin{bmatrix}
-2&0&\framebox{-3}\\
0&10&0\\
-3&0&-2\\
\end{bmatrix}}z^{-1}+
{\scriptsize\begin{bmatrix}
0&0&-3&0&-2\\
0&10&0&40&0\\
-3&0&\framebox{-84}&0&-3\\
0&40&0&10&0\\
-2&0&-3&0&0\\
\end{bmatrix}}z+
{\scriptsize\begin{bmatrix}
-2&0&-3\\
0&10&0\\
\framebox{-3}&0&-2\\
\end{bmatrix}}z^{3}.$$

By Proposition~\ref{3.3}, $H_{1,0,1}(x,y,z)=H_{1,1,0}(x,z,y)$ and $H_{0,1,1}(x,y,z)=H_{1,1,0}(z,y,x)$.
We observe that $|\supp H_{1,1,0}|=|\supp H_{1,0,1}|=|\supp H_{0,1,1}|=21$.

Finally,
\begin{align*} 
H_{1,1,1} & = w_{0,0,0}xyz - w_{1,1,1}\\ 
&= -3x^2y^2z^2 - 2x^2y^2 - 2x^2z^2 - 2y^2z^2 + 45xyz - 3x^2 - 3y^2 - 3z^2 + 5xyz^{-1}\\
&\quad + 5xy^{-1}z + 5x^{-1}yz - 84 +5xy^{-1}z^{-1} + 5x^{-1}yz^{-1} + 5x^{-1}y^{-1}z - 3z^{-2} - 3y^{-2}\\
&\quad - 3x^{-2} + 45x^{-1}y^{-1}z^{-1} - 2y^{-2}z^{-2} -2x^{-2}z^{-2} - 2x^{-2}y^{-2} - 3x^{-2}y^{-2}z^{-2},
\end{align*}
with
\begin{align*}
\cH_{1,1,1} &=
{\scriptsize\begin{bmatrix}%
-2&0&\framebox{-3}\\
0&0&0\\
-3&0&-2\\
\end{bmatrix}}z^{-2}+
{\scriptsize\begin{bmatrix}
5&0&5\\
0&\framebox{0}&0\\
45&0&5\\
\end{bmatrix}}z^{-1}
+{\scriptsize\begin{bmatrix}
0&0&-3&0&-2\\
0&0&0&0&0\\
-3&0&\framebox{-84}&0&-3\\
0&0&0&0&0\\
-2&0&-3&0&0\\
\end{bmatrix}}z^0\\
&\quad+{\scriptsize\begin{bmatrix}
5&0&45\\
0&\framebox{0}&0\\
5&0&5\\
\end{bmatrix}}z^1
+{\scriptsize\begin{bmatrix}
-2&0&-3\\
0&0&0\\
\framebox{-3}&0&-2\\
\end{bmatrix}}z^2.
\end{align*}
We have
$|\supp H_{1,1,1} | =  23$. Thus, $N=155$. 

\section*{Acknowledgements}
The second author has received funding through the European Union's MSCA4Ukraine project (ID 1232926).

\bibliographystyle{abbrv} 
\bibliography{wavelets}

\end{document}